\documentclass[11pt]{amsart}

\def\Prob{\operatorname{Prob}}

\def\dxb{\underline x_b}
\def\gxb{\overline x_b}
\def\da{\underline a}
\def\ga{\overline a}
\newtheorem{theorem}{Theorem}
\newtheorem{proposition}[theorem]{Proposition}
\newtheorem{lemma}[theorem]{Lemma}

\theoremstyle{remark}
\newtheorem{remark}[theorem]{\bf Remark}
\usepackage{epsfig} 
\usepackage{graphicx} 
\usepackage{enumerate} 

\begin{document}

\title{Cell cycle length and long-time behaviour of an age-size model}
\author[K. Pich\'or]{Katarzyna Pich\'or}
\address{K. Pich\'or, Institute of Mathematics,
University of Silesia, Bankowa 14, 40-007 Kato\-wi\-ce, Poland.}
\email{katarzyna.pichor@us.edu.pl}
\author[R. Rudnicki]{Ryszard Rudnicki}
\address{R. Rudnicki, Institute of Mathematics,
Polish Academy of Sciences, Bankowa 14, 40-007 Katowice, Poland.}
\email{rudnicki@us.edu.pl}
\date{March 19th, 2021}

\begin{abstract}
We consider an age-size structured cell population model based on the cell cycle length.
The model is  described by a first order partial differential equation with initial-boundary conditions.
Using the theory of semigroups of  positive operators we establish new criteria for an
asynchronous exponential growth of solutions to such equations.
We discuss the question of exponential size growth of cells.
We study in detail a constant size growth model and a model with target size  division.
We also present versions of the model when the population is heterogeneous.
\end{abstract}

\maketitle

\section{Introduction}
\label{intro}
The cell cycle is a series of events that take place in a cell leading to its replication. It is regulated by a
complex network of protein interactions~\cite{Morgan}.
Modern experimental techniques concerning the cell cycle \cite{CSK,I-B,MC,Perego,T-A,Taniguchi,Vittadello-exp,Wang} allow us not only to understand processes inside single cells, 
but also to build more precise cellular populations models.   

Most of populations are usually heterogeneous. Thus it is important to consider the distribution of the population 
according to some significant parameters such as age, size, maturity, or
proliferative state of cells.  Models of this type are called structured. This type of models are usually represented by partial differential equations 
with some nonlocal perturbations and specific boundary conditions.
Knowing the length of cell cycle  allows us to predict  the development of unicellular populations
and tissues growth and maintenance.

The aim of the paper is twofold. Firstly,  to construct an age-size structured model assuming that we know the growth of individual cells and the distribution
of the  cell cycle length.  Secondly, to study the long-time behaviour of the solution of this model. 

We consider a model which is based on the following assumptions.
The population  grows in steady-state conditions. 
Cells can be described by their age $a$ and size $x$ alone and reproduction occurs by fission into two equal parts.
The distribution of the cell cycle length depends only on the initial size $x_b$ of a cell.
The velocity of growth of an individual cell depends only on its size $x$, i.e. $x'(t)=g(x(t))$.
We also assume that sizes of cells and cell cycle durations are bounded above and bounded away from zero.
Moreover, we assume that the initial daughter cell sizes are distributed in some interval which contains the mother initial size.
We formulate a mathematical model which describes the time evolution of  distribution of cellular age and size.
The model consists of a  partial differential equation with an integral boundary condition and an initial condition.
The novelty of our model is that we use the distribution of the length of the cell cycle, instead of 
a size dependent  probability of division  usually used in size-structured models \cite{BA,DHT,Doumic,GH,GW,Heijmans}.  Such probability is difficult to measure experimentally in contrast to the length of the cell cycle.
 
We check that the solutions of our model generate a continuous semigroup of operators $\{U(t)\}_{t\ge 0}$  on some $L^1$ space. 
Under additional assumption that  $g(2x)\ne 2g(x)$ for some $x$, we prove that the semigroup  $\{U(t)\}_{t\ge 0}$  
has asynchronous exponential growth (AEG), i.e. 
\begin{equation}
\label{AEG1}
e^{-\lambda t}U(t)u_0(x_b,a)\to  Cv(x_b,a)\quad \textrm{for $t\to\infty$},
\end{equation}
where $\lambda$ is the Malthusian parameter and $v$ is a stable initial size and age distribution,
 which does not depend on the initial distribution $u_0$.
The property AEG plays an important role in the study of structured population models \cite{ASW,DHT,GH,Webb-cc},
 because we can expect that the real process should be close to a stationary
state and then it is easy to estimate biological parameters~\cite{LRB}.

The proof of AEG of $\{U(t)\}_{t\ge 0}$ is based on the reduction of  the problem to a stochastic (Markov) semigroup \cite{LiM} by using the Perron eigenvectors and on the theorem that a partially integral stochastic semigroup having a unique invariant density is asymptotically stable 
\cite{PR-jmaa2}.
A similar technique was applied to study  other population models \cite{BPR,Pichor-MCM,RP} 
and to some piecewise deterministic Markov processes  \cite{Mac-Tyr,PR-cell-cyc,RT-K-k}.
 We note that AEG property can be proved by using known results on
compact semigroups but it seems to be difficult to check compactness
and analyze the spectrum of the generator of our semigroup.
It is interesting that even  nonlinear models of cell population (cf. \cite{M-R,RP-M}) can be
reduced to stochastic semigroups.

The last two sections contain corollaries from our results (Section~\ref{s:remarks}) and
some remarks concerning other models and experimental data (Section~\ref{s:other2}).
One of the main points of these sections is what can happen when $g(2x)=2g(x)$ for all $x$\,?
This is an important question because it is usually assumed that the size (volume) of a cell grows exponentially,
which means that $g(x)=\kappa x$ and in this case $g(2x)=2g(x)$. If we include in a model the assumption that
$g(x)=\kappa x$, then we can obtain some paradoxical results. For example, if the size discrepancy between newborn cells is small, then 
the descendants of one cell can have the same size at the same time and the size of the population does not grow exponentially even in 
steady-state conditions. Of course the law of exponential size growth is statistical in nature and we can modify it by considering some random fluctuations 
in the growth rate. Another  problem considered in Section~\ref{s:other2} is how to incorporate into our description 
some models of the cell cycle: a \textit{constant $\Delta$ model} and a \textit{model with target size division}.
Finally, we present versions of the model when the population is heterogeneous, e.g. with an asymmetric division or
with fast and slow proliferation.

\section{Model}
\label{s:model}
We consider the following model of the cell cycle. Denote, respectively, by $a$, $x_b$, and $x$ ---  the age, the initial size, and the  size  of a  cell.
We assume that $\dxb$ and $\gxb$   are the minimum and maximum sizes of newborn cells.
We also assume that cells age with unitary velocity and grow  with a velocity $g(x)$,
i.e. if a cell has the initial size $x_b$, then the size at age $a$ satisfies the equation
\begin{equation}
\label{grow}
x'(a)=g(x(a)),\quad x(0)=x_b.
 \end{equation}
We denote by $\pi_ax_b$ the  solution of (\ref{grow}).
The length $\tau$ of the cell  cycle is a random variable which depends on the initial cell size $x_b$; has values in some interval 
$[\da(x_b),\ga(x_b)]$;  and has the probability density distribution $q(x_b,a)$, i.e.
the integral $\int_0^{A} q(x_b,a)\, da$ is the probability that $\tau\le A$.  
According to the definition of $q$, if a cell has the initial size $x_b$, then $\Phi(x_b,a)=\int_a^{\infty} q(x_b,r)\,dr$ is its \textit{survival function},
i.e. $\Phi(x_b,a)$ is the probability that a cell will not split before age $a$. 
We assume that if the mother cell has size $x$ at the moment of division, then the daughter cells have size $x/2$, i.e. if the initial size of the mother cell is $x_b$ and $\tau=a$ is the length of its cell cycle, then the initial size of the daughter cell is $S_{a}(x_b)=\tfrac12\pi_ax_b$.

Now we collect the assumptions concerning the functions $g$ and $q$ used in the paper:
\vskip1mm

\noindent (A1) $g\colon [\dxb,2\gxb]\to (0,\infty)$
 is a $C^1$-function,

\noindent (A2) $q\colon [\dxb,\gxb]\times [0,\infty)\to [0,\infty)$ is a continuous function and for each $x_b$ the function $a\mapsto q(x_b,a)$ is a probability density,

\noindent (A3)  $0<\da(x_b)<\ga(x_b)<\infty$,  $q(x_b,a)>0$ if $a\in  (\da(x_b),\ga(x_b))$, 
and  $q(x_b,a)=0$ if $a\notin (\da(x_b),\ga(x_b))$ for each $x_b\in [\dxb,\gxb]$, 

\noindent (A4) $x_b\mapsto \da(x_b)$ and $x_b\mapsto \ga(x_b)$ are continuous functions,

\noindent (A5)  $S_{\da(x_b)}(x_b)\ge \dxb$ and $S_{\ga(x_b)}(x_b)\le \gxb$ for each $x_b\in [\dxb,\gxb]$,

\noindent (A6) $S_{\da(x_b)}(x_b)<x_b<S_{\ga(x_b)}(x_b)$ for each $x_b\in (\dxb,\gxb)$. 

Fig.~\ref{r:cell-cyc1} and Fig.~\ref{r:cell-cyc2} illustrate our assumptions. Only assumption (A6) needs some explanation.
We assume that a daughter cell can have the same initial size as the initial size of a mother cell.
In Section~\ref{s:asyp-beh} we will add an extra assumption (A7) which will be used only to show 
the long-time behaviour of the distribution of $(x_b,a)$.
\begin{figure}
\centerline{\includegraphics{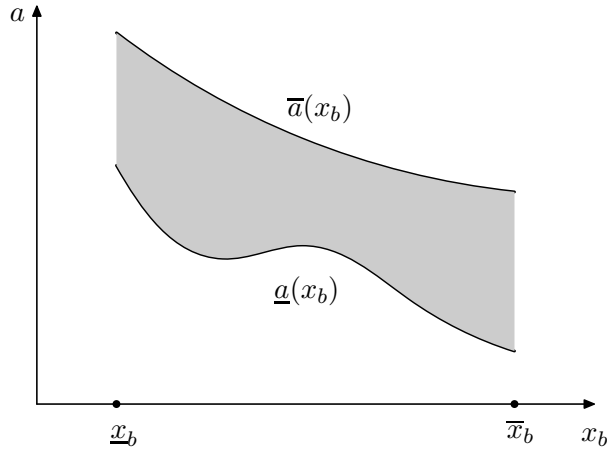}}
\centerline{
\begin{picture}(210,0)(0,0)
\put(28,2){$\underline x_b$}
\put(178,2){$\overline x_b$}
\put(206,1){$x_b$}
\put(95,124){$\overline a(x_b)$}
\put(90,55){$\underline a(x_b)$}
\put(-10,160){$a$}
\end{picture}
}
\vskip2pt
\caption{An example of functions 
$x_b\mapsto \da(x_b)$ and $x_b\mapsto \ga(x_b)$. 
The function $q$ is positive between the graphs of these functions.}
\label{r:cell-cyc1}
\end{figure}

\begin{figure}
\centerline{\includegraphics{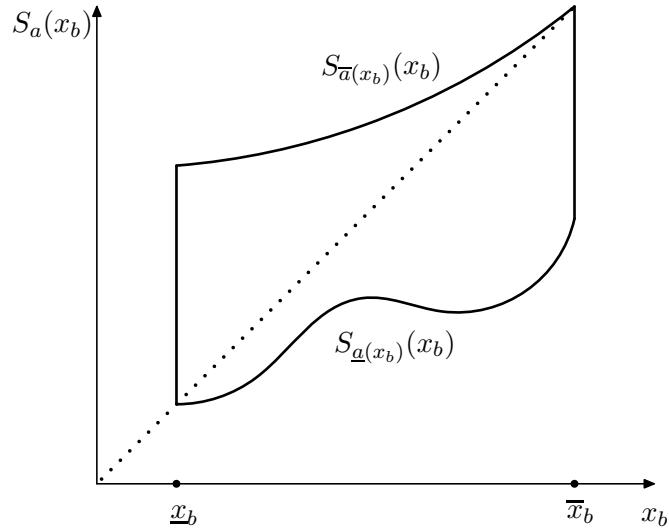}}
\centerline{
\begin{picture}(210,0)(0,0)
\put(28,2){$\underline x_b$}
\put(178,2){$\overline x_b$}
\put(206,1){$x_b$}
\put(85,170){$S_{\overline a(x_b)}(x_b)$}
\put(90,65){$S_{\underline a(x_b)}(x_b)$}
\put(-32,186){$S_a(x_b)$}
\end{picture}
}
\vskip2pt
\caption{The relation between the initial sizes of mother and daughter cells (A5,A6).}
\label{r:cell-cyc2}
\end{figure}

Assume that a cell with initial size $x_b$ and age $a$ splits in the  time interval of the length $\Delta t$  
with probability  $p(x_b,a)\Delta t+o(\Delta t)$, i.e.
\[
p(x_b,a)=\lim_{\Delta t\downarrow 0}
\frac{\operatorname P(\tau\in [a,a+\Delta t] \mid \tau\ge a)}{\Delta t}.
\] 
Since $\Phi(x_b,a)=\exp\big(-\int_0^a p(x_b,r)\,dr\big)$,  an easy computation shows
that
\begin{equation*}
\begin{aligned}
q(x_b,a)&=p(x_b,a)\exp\big(-\textstyle{\int_0^a} p(x_b,r)\,dr\big),\\
\quad p(x_b,a)&=\frac{q(x_b,a)}{\int_a^{\infty} q(x_b,r)\,dr} 
\end{aligned}
 \end{equation*}
for $a<\ga(x_b)$. As   $\Phi(x_b,\ga(x_b))=0$, we have $\int_0^{\ga(x_b)}p(x_b,a)\,da=\infty$.


In order to derive a master equation for the distribution of the population 
with respect to $x_b$ and $a$ we need to introduce a family of Frobenius-Perron operators which describe the relation between the initial sizes of mother and daugther cells.

Let $f(x_b)$  be the density of initial sizes of mother cells that have the fixed length of cell cycle $\tau=a$ for some $a\in (\da,\ga)$,
where $\da=\min\da(x_b)$ and $\ga=\max\ga(x_b)$.
Denote by $P_af(x_b)$ the density of  initial sizes of daughter cells.
\begin{lemma}
Let $x_a$ be the minimum initial size of cells which can split at age $a$. Then
\begin{equation}
\label{F-P7}
P_af(x_b)=\frac{2g(\pi_{-a}(2x_b))}{g(2x_b)}f(\pi_{-a}(2x_b))\mathbf 1_{[S_a(x_a),\gxb]}(x_b).
\end{equation}
\end{lemma}
\begin{proof}
Observe that   
$x_a=\dxb$ if $S_a(\dxb)\ge \dxb$ or $x_a=S_a^{-1}(\dxb)$ otherwise. 
It is clear that $q(x_b,a)=0$ for $x_b< x_a$.
We have 
\begin{equation*}
\int_{S_a(x_a)}^{S_a(y)}P_af(x_b)\,dx_b=\int_{x_a}^{y}f(x_b)\,dx_b\quad \textrm{for $y\ge x_a$}
 \end{equation*}
or, equivalently,  
\begin{equation}
\label{F-P2}
\int_{S_a(x_a)}^{x}P_af(x_b)\,dx_b=\int_{x_a}^{S_a^{-1}(x)}f(x_b)\,dx_b\quad \textrm{for $x\ge S_a(x_a)$}.
 \end{equation}
From (\ref{F-P2}) it follows that  
\begin{equation*}
P_af(x_b)=\frac{d}{dx_b}\big(S_a^{-1}(x_b)\big)f\big(S_a^{-1}(x_b)\big)\mathbf 1_{[S_a(x_a),\gxb]}(x_b).
 \end{equation*}
Using the formula $S_a^{-1}(x_b)=\pi_{-a}(2x_b)$ we check that    
\begin{equation}
\label{F-P4}
\frac{d}{dx_b}\big(S_a^{-1}(x_b)\big)=\frac{2g(\pi_{-a}(2x_b))}{g(2x_b)}.
 \end{equation}
In order to show \eqref{F-P4} we introduce two functions:
\[
\varphi(x_b,a)=\pi_{-a}(2x_b)\quad\textrm{and}\quad
\psi(x_b,a)=\frac{\partial \varphi}{\partial x_b}(x_b,a).
\]
From $\varphi(x_b,0)=2x_b$ we obtain $\psi(x_b,0)=2$. 
Since $\frac{\partial \varphi}{\partial a}(x_b,a)=-g(\varphi(x_b,a))$,
we have
\begin{align*}
\frac{\partial \psi}{\partial a}(x_b,a)
&=\frac{\partial }{\partial a}
\frac{\partial \varphi}{\partial x_b}(x_b,a)
=\frac{\partial }{\partial x_b}\frac{\partial \varphi}{\partial a}
(x_b,a)\\
&=\frac{\partial }{\partial x_b}\big(-g(\varphi(x_b,a))\big)
=-g'(\varphi(x_b,a))\psi(x_b,a).
\end{align*}
We have received the linear equation $\partial\psi/\partial a= -g'(\varphi(x_b,a))\psi$ with the initial condition $\psi(x_b,0)=2$
which has the solution
\begin{equation*}
\psi(x_b,a)=2\exp\bigg(-\int_0^ag'(\varphi(x_b,r))\,dr  \bigg).
 \end{equation*}
Substituting $y=\varphi(x_b,r)$ we receive $dy/dr=-g(y)$ and 
\begin{equation*}
\psi(x_b,a)=2\exp\bigg(\int_{2x_b}^{\pi_{-a}(2x_b)} \frac{g'(y)}{g(y)}\,dy  \bigg)=\frac{2g(\pi_{-a}(2x_b))}{g(2x_b)}. \qedhere
 \end{equation*}
 \end{proof}
The formula \eqref{F-P7} defines a family of operators $P_a\colon L^1[\dxb,\gxb]\to L^1[\dxb,\gxb]$, $a\ge 0$.
The operators $P_a$ are well defined for $a\in (\da,\ga)$ but we extend the definition of $P_a$ setting  $P_af\equiv 0$ for others $a$'s. 
For each $a$ the operator $P_a$ is linear and \textit{po\-si\-tive}, i.e. if $f\ge 0$, then $P_af\ge 0$.
Moreover  $\|P_af\|_{L^1}\le \|f\|_{L^1}$.
The adjoint operator of $P_a$
acts on the space $L^{\infty}[\dxb,\gxb]$
and it is given by $P_a^*f(x_b)=f(S_a(x_b))=f(\frac12\pi_ax_b)$ for $x_b\ge x_a$ and $P_a^*f(x_b)=0$ for $x_b< x_a$. 

We denote by $u(t,x_b,a)$ the number of individuals in a population  having initial size $x_b$ and age $a$ at time $t$.
Then,  according to our assumptions concerning the model, $p(x_b,a)u(t,x_b,a)\Delta t$ is the number of cells of initial size $x_b$ and age $a$ which split in a time interval of the length $\Delta t$. It means that $2\Delta t\int_0^{\infty}\Big(P_a\big(p(\cdot,a)u(t,\cdot,a)\big)\Big)(x_b)\,da $ is the number of new born cells in this time interval. 
It should be noted that the operator $P_a$ in the last integral acts on the function $\psi(x_b)=p(x_b,a)u(t,x_b,a)$ at fixed values $t$ and $a$.  
If there are no limitations concerning the growth of the population and all cells split, 
then the function $u$ satisfies the following initial-boundary problem:
\begin{align}
\label{eq1}
&\frac{\partial u}{\partial t}(t,x_b,a)
 +\frac{\partial u}{\partial a}(t,x_b,a)
 =-p(x_b,a)u(t,x_b,a),\quad  a<\ga(x_b),\\
&u(t,x_b,0)=2\int_0^{\infty}\Big(P_a\big(p(\cdot,a)u(t,\cdot,a)\big)\Big)(x_b)\,da, 
\label{eq2}\\
&u(0,x_b,a)=u_0(x_b,a). 
\label{eq3}
\end{align}

We assume that $u_0$ is a nonnegative measurable function such that
\begin{equation}
\label{def-u_0}
\int_{\dxb}^{\gxb}\int_0^{\infty}
u_0(x_b,a)\Psi(x_b,a)\,da\,dx_b<\infty,
\end{equation} 
where $\Psi(x_b,a)=\exp\big(\int_0^a p(x_b,\bar a)\,d\bar a\big)$.
In  (\ref{def-u_0}) we have assumed that the initial condition $u_0$ is integrable with weight 
$\Psi(x_b,a)=\Phi(x_b,a)^{-1}$ because  $\Phi(x_b,a)$ is the fraction of cells which will survive beyond age $a$.  
Since $\ga(x_b)$ is the maximum age of a cell with initial size $x_b$, 
it is reasonable to consider variables $x_b$ and $a$ only from the set 
\[
X=\{(x_b,a)\colon   \dxb\le x_b\le\gxb,\,\,\, 0\le a\le\ga(x_b)\}
\]
(see Fig.~\ref{r:cell-cyc3}).
\begin{figure}
\centerline{\includegraphics{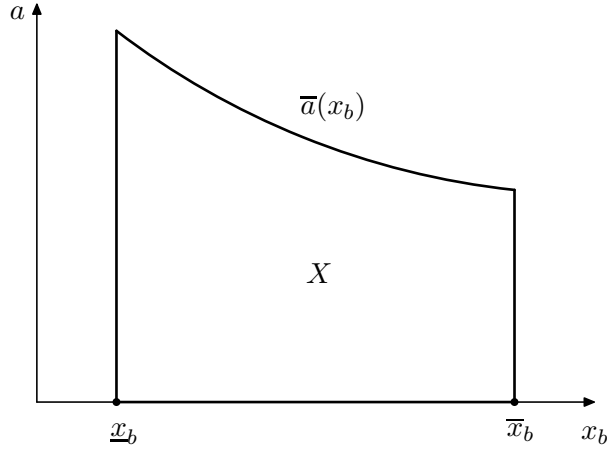}}
\centerline{
\begin{picture}(210,0)(0,0)
\put(28,2){$\underline x_b$}
\put(178,2){$\overline x_b$}
\put(206,1){$x_b$}
\put(100,124){$\overline a(x_b)$}
\put(-10,160){$a$}
\put(102,60){$X$}
\end{picture}
}
\vskip2pt
\caption{The set $X$.}
\label{r:cell-cyc3}
\end{figure}
Though we consider $a\le\ga(x_b)$, it will be convenient to keep the notation of integral  $\int_0^{\infty}$ with respect to $a$ as in 
formula (\ref{eq2}) assuming that $u(t,x_b,a)=0$ for $a> \ga(x_b)$.

Let $\mathcal B(X)$ be the $\sigma$-algebra of Borel subsets of $X$, 
$\ell$ be the Lebesgue measure on $X$, and $E$ be 
the space $L^1(X)=L^1(X,\mathcal B(X),\ell)$. By $\|\cdot\|_E$ we denote the  norm in $E$. 

The main purpose of the paper is to show that the solutions of the system~\eqref{eq1}--\eqref{eq3} have asynchronous
exponential growth \eqref{AEG1}. 
The AEG property can be written in the following way:
\begin{equation}
\label{AEG11}
\lim_{t\to\infty} e^{-\lambda t}u(t,x_b,a)=\alpha(u_0)v(x_b,a),
\end{equation}
where the limit is in the space $E$, $\alpha$ is a linear and bounded functional on $E$, and $v\in E$ does not depend on $u_0$
(see Theorem~\ref{th:long-time-u}).
We prove this fact under an additional assumption that  $g(2x)\ne 2g(x)$ for some $x$.
The schedule of the proof is the following. In Section~\ref{exist-sol} 
we replace the system~\eqref{eq1}--\eqref{eq3} by one in which the first equation of the system has zero on the right-hand side. 
Then we construct a $C_0$-semigroup of positive operators $\{T(t)\}_{t\ge 0}$ on the space $L^1(X)$ corresponding to the new system.
In Section~\ref{s:eigen} we prove that the infinite\-simal generator $\mathcal A$ of the semigroup 
$\{T(t)\}_{t\ge 0}$ and the adjoint  operator $\mathcal A^*$ have positive 
 eigenvectors, $f_i$ and $v$, respectively,  for some eigenvalue $\lambda$.
In Section~\ref{s:asyp-beh} we introduce a
semigroup $\{P(t)\}_{t\ge 0}$ given by $P(t)f=e^{-\lambda t}T(t)f$ defined on the space $E_1=L^1(X,\mathcal B(X),\mu)$
with the measure $\mu$ given by $d\mu=v\, dx_bda$.
We check that  $\{P(t)\}_{t\ge 0}$ is a stochastic semigroup on $E_1$ and  that $f_i$ is the unique invariant density 
of $\{P(t)\}_{t\ge 0}$. We also formulate some general theorem concerning asymptotic stability of stochastic semigroups.
We apply this theorem to the semigroup $\{P(t)\}_{t\ge 0}$ and prove its asymptotic stability. 
We translate this result in terms of the semigroup $\{T(t)\}_{t\ge 0}$.
Finally we return to the semigroup $\{U(t)\}_{t\ge 0}$ generated by the system~\eqref{eq1}--\eqref{eq3}
and we show that it has the AEG property.

Most of experiments concerning microorganisms are conducted in chemo\-stats, where cells can be grown in a physiological steady state under constant environmental conditions. Then cells are removed from the system with the outflow with rate $D$. In this case we add to the right-hand side of equation (\ref{eq1}) the term $-Du(t,x_b,a)$. Similarly, if cells die with rate $d(t,x_b,a)$, then we add to the right-hand side of  (\ref{eq1})  the term $-d(t,x_b,a)u(t,x_b,a)$.   
One can consider more advanced models with cellural competition, but in this case all functions $q$, $g$, and $d$ can depend also on the total number of cells 
and we do not investigate such models.

\section{A semigroup approach}
 \label{exist-sol}
It is convenient to substitute $z(t,x_b,a)=u(t,x_b,a)\Psi(x_b,a)$  and  $z_0(x_b,a)=u_0(x_b,a)\Psi(x_b,a)$ in (\ref{eq1})--(\ref{eq3}).
Then the system (\ref{eq1})--(\ref{eq3}) takes the form 
\begin{align}
\label{eq1-n}
&\frac{\partial z}{\partial t}(t,x_b,a)
 +\frac{\partial z}{\partial a}(t,x_b,a) =0,\quad  a<\ga(x_b),\\
&z(t,x_b,0)=2\int_0^{\infty}\Big(P_a\big(q(\cdot,a)z(t,\cdot,a)\big)\Big)(x_b)\,da, 
\label{eq2-n}\\
&z(0,x_b,a)=z_0(x_b,a). 
\label{eq3-n}
\end{align}
Observe that $z_0$ is an integrable function.
Since the expression on the right-hand side of (\ref{eq2-n}) will be used quite often, so instead of it, we will use the shortened notation $\mathcal Pz(t,x_b)$. Thus, equation (\ref{eq2-n}) takes the form $z(t,x_b,0)=\mathcal Pz(t,x_b)$. We also use a simplified  notation $\mathcal Pf(x_b)$ for the expression $2\int_0^{\infty}\Big(P_a\big(q(\cdot,a)f(\cdot,a)\big)\Big)(x_b)\,da$.

We consider the solutions of (\ref{eq1-n})--(\ref{eq3-n}) as continuous functions $z\colon [0,\infty)\to E$ 
defined by $z(t)(x_b,a)=z(t,x_b,a)$, where $z(t)$ is the solution of the evolution 
\[
z'(t)=\mathcal Az(t),\quad z(0)=z_0,
\]
with the operator
\[
\mathcal Af(x_b,a)=-\frac{\partial f}{\partial a}(x_b,a)
 \]
which has the domain
 \[
 \mathcal D(\mathcal A)
=\Big\{f\in E,\quad \frac{\partial f}{\partial a}\in E,\quad
f(x_b,0)=\mathcal Pf(x_b)\Big\}.
\]

Since a function $f\in E$ is only almost everywhere defined we need to clarify the formula for $f(x_b,0)$.
Consider the (partial) Sobolev space 
\[
W_1(X)=\Big\{f\in E\colon \frac{\partial f}{\partial a}\in E\Big\}
\]   
with the norm $\|f\|_{W_1(X)}=\|f\|_E+\Big\|\dfrac{\partial f}{\partial a}\Big\|_E$. 
In the space $W_1(X)$ we introduce a trace operator $\mathcal T\colon W_1(X)\to L^1[\dxb,\gxb]$,  $\mathcal Tf(x_b)=f(x_b,0)$, 
in the following way. First we show that there exists a constant $c>0$ such that 
for each function $f\in  W_1(X)\cap C(X)$ we have 
\[
\|\mathcal Tf\|_{L^1[\dxb,\gxb]}\le c\|f\|_{W_1(X)}
\]
(see \cite{Evans} Theorem 1, Chapter 5.5).
Since the set $W_1(X)\cap C(X)$ is dense in $W_1(X)$ we can extend $\mathcal T$ uniquely  to a linear bounded operator
on the whole space $W_1(X)$. 
\begin{proposition}
\label{D(A)}
The operator $\mathcal A$ with the domain
\[
\mathcal D(\mathcal A)=\big\{f\in W_1(X)\colon \mathcal Tf =\mathcal Pf \big\}
\]
generates a positive $C_0$-semigroup $\{T(t)\}_{t\ge 0}$ on $E$.
\end{proposition}
We recall that a family $\{T(t)\}_{t\ge0}$ of linear
operators on a Banach space $E$
is a $C_0$-\textit{semigroup}
or \textit{strongly continuous semigroup}
if it satisfies the following conditions:
\begin{enumerate}[\rm(a)]
\item \ $T(0)=I$,  i.e., $T(0)f =f$ for $f\in E$,
\item \ $T(t+s)=T(t) T(s)\quad \textrm{for}\quad
s,\,t\ge0$,
\item \  for each $f\in E$ the function
$t\mapsto T(t)f$ is continuous.
\end{enumerate}

The proof of this result is based on a perturbation method related to
operators with boundary conditions developed in \cite{greiner}
and an extension of this method to unbounded perturbations in $L^1$ space
in \cite{GMTK} (see Theorem~\ref{MTK}  below).

\begin{theorem}
\label{MTK}
Let $(\Gamma,\Sigma,m)$, $(\Gamma_{\partial},\Sigma_{\partial},m_{\partial})$ be $\sigma$-finite measure spaces
and let $L^1=L^1(\Gamma,\Sigma,m)$ and $L_{\partial}^1=L^1(\Gamma_{\partial},\Sigma_{\partial},m_{\partial})$.
Let $\mathcal{D}$ be a linear subspace of~$L^1$.
We assume that
$A\colon \mathcal D\to L^1$ and $\Upsilon_0,\Upsilon\colon \mathcal D\to L_{\partial}^1$ are linear operators satisfying the following conditions: 
\begin{enumerate}[\rm(1)]
\item
for each $\lambda>0$, the operator $\Upsilon_0\colon \mathcal{D}\to L^1_{\partial}$ restricted to the nullspace 
$\mathcal N(\lambda I-A)=\{f\in\mathcal D\colon \lambda f- Af=0\}$ has a  positive right inverse $\Upsilon(\lambda)\colon L^1_{\partial}\to \mathcal N(\lambda I-A)$, i.e.  
$\Upsilon_0\Upsilon(\lambda)f_{\partial}=f_{\partial}$ for $f_{\partial} \in L^1_{\partial}$; 
\item
 the operator $\Upsilon\colon \mathcal{D} \to L^1_{\partial}$ is positive and there is $\omega>0$ such that $\|\Upsilon \Upsilon(\lambda)\|<1$ for $\lambda>\omega$;
\item the operator  $A_0=A\big|_{\mathcal D(A_0)} $, where $\mathcal{D}(A_0)=\{f\in \mathcal{D}\colon  \Upsilon_0f=0\}$,
 generates a positive $C_0$-semigroup on $L^1$;
\item
$\int_{\Gamma} Af(x)\,m(dx)\le \int_{\Gamma_{\partial}} \Upsilon_0 f(x_\partial)\,m_{\partial}(dx_\partial)$
for $f\in \mathcal{D}_+=\{f\in \mathcal{D}\colon f\ge 0\}$.
\end{enumerate}
Then $A$ with the domain $\mathcal{D}(A)=\{f\in \mathcal{D}\colon  \Upsilon_0 f=\Upsilon f\} $  generates a  positive  semigroup on $L^1$.
\end{theorem}

\begin{proof}[Proof of Proposition~$\ref{D(A)}$]
First we translate our notation to that from Theorem~\ref{MTK}.
Let $\Gamma=X$, $\Gamma_{\partial}=[\dxb,\gxb]$, $L^1=E$,  $L_{\partial}^1=L^1[\dxb,\gxb]$, 
$\Upsilon_0=\mathcal T$, $\Upsilon =\mathcal P$, $Af=-\frac{\partial f}{\partial a}$ and
$\mathcal D=\Big\{f\in E\colon\,\, \frac{\partial f}{\partial a}\in E\}$.

(1): Since $Af=-\frac{\partial f}{\partial a}$, the nullspace  $\mathcal N(\lambda I-A)$ is the set of functions $f\in \mathcal D$ satisfying equation
 \[
\frac{\partial f}{\partial a} +\lambda f=0.
\]
Solving this equation we obtain that $f(x_b,a)=f(x_b,0)e^{-\lambda a}$.
Thus the operator $\Upsilon_0$ restricted to $\mathcal N(\lambda I-A)$ is invertible 
and the inverse operator $\Upsilon(\lambda)\colon L^1[\dxb,\gxb] \to \mathcal N(\lambda I-A)$ given by 
$\Upsilon(\lambda)f(x_b,a) =f(x_b)e^{-\lambda a}$ is positive.

(2): Since $\Upsilon=\mathcal P$  we  check that 	$\|\mathcal P\Upsilon(\lambda)\|<1$ for $\lambda>\omega=\ln2/\da$.
Take $f\in L^1[\dxb,\gxb]$, $f\ge 0$,  and  let  $\Theta_{\lambda}(x_b,a)=2e^{-\lambda a}q(x_b,a)$.
Then 
\begin{align*}
\int_{\dxb}^{\gxb}(\mathcal P\Upsilon(\lambda)f)(x_b)\,dx_b&=\int_{\dxb}^{\gxb}\int_0^{\infty} P_a (f(\cdot)\Theta_{\lambda}(\cdot,a))(x_b)\,da\,dx_b\\
&\le  \int_{\dxb}^{\gxb}\int_0^{\infty}f(x_b)\Theta_{\lambda}(x_b,a)\,da\,dx_b.
\end{align*}
Since
\[
\int_0^{\infty}\Theta_{\lambda}(x_b,a)\,da=\int_{\da}^{\ga}2e^{-\lambda a}q(x_b,a)\,da\le 2e^{-\lambda \da}<1
\]
for $\lambda>\omega=\ln2/\da$, we have $\|\mathcal P\Upsilon(\lambda)\|<1$ for $\lambda>\omega$.

(3): The operator $A_0$ generates a positive $C_0$-semigroup $\{T_0(t)\}_{t\ge 0}$ on $E$ given by 
\[
T_0(t)f(x_b,a)=
\begin{cases}
f(x_b,a-t)\quad\textrm{for $a> t$,}\\
0 \quad\textrm{for $a<t$}.
\end{cases}
\]

(4): If $f\in \mathcal{D}_+$ then
\[
\int_X Af(x_b,a)\,dx_b\,da
=-\int_X\frac{\partial f}{\partial a}(x_b,a)  \,da \,dx_b
=\int_{\dxb}^{\gxb}\mathcal Tf(x_b)\,dx_b.
\]
Since $\Upsilon_0=\mathcal T$, $\Gamma=X$, and $\Gamma_{\partial}=[\dxb,\gxb]$ we have 
\[
\int_X Af(x_b,a)\,dx_b\,da-\int_{\dxb}^{\gxb} \Upsilon_0 f(x_b,a)\,dx_b=0.\qedhere
\]
\end{proof}

\begin{remark}
\label{aler-proof}
It is not difficult to check that 
the resolvent $R(\lambda,\mathcal A)=(\lambda I-\mathcal A)^{-1}$ exists for sufficiently large $\lambda>0$ and it is given by
the formula
\begin{equation}
\label{R(l,a)}
R(\lambda,\mathcal A)=(I-\mathcal P_{\lambda})^{-1}R(\lambda,A_0),
\end{equation}
where 
 \[
\mathcal P_{\lambda}f(x_b,a)=e^{-\lambda a}\mathcal Pf(x_b),
\quad
R(\lambda,A_0)f(x_b,a)=\int_0^af(x_b,r)e^{\lambda(r-a)}\,dr
\]
for $f\in E$. 
Another proof of Proposition~$\ref{D(A)}$ can be done directly using formula \eqref{R(l,a)} and the Hille--Yosida theorem.
\end{remark}

\begin{remark}
\label{classical-solution}
If $z_0\in \mathcal D(\mathcal A)$ and $z(t)=T(t)z_0$, then $z(t)\in \mathcal D(\mathcal A)$, 
$z'(t)$ exists and $z'(t)=\mathcal Az(t)$.
Assume that the function $z_0$ and the derivative $\dfrac{\partial z_0}{\partial a}$ are continuous bounded functions   
and the consistency condition  $z_0(x_b,0)=\mathcal Pz_0(x_b)$ holds.
Then the problem (\ref{eq1-n})--(\ref{eq3-n}) has a 
unique classical solution. By the \textit{classical solution} we understand a continuous function $z$, which has     
continuous derivatives $\dfrac{\partial z}{\partial a}$ and $\dfrac{\partial z}{\partial t}$ outside the set 
$\mathcal Z=\{(t,x_b,a)\colon a=t,\,(x_b,a)\in X\}$,
$z$ satisfies (\ref{eq1-n}) outside  $\mathcal Z$, and $z$ satisfies conditions (\ref{eq2-n})--(\ref{eq3-n}).
In this case $z_0\in \mathcal D(\mathcal A)$ and $z(t,x_b,a)= T(t)z_0(x_b,a)$, i.e. 
classical and ``semigroup" solutions are identical. 
\end{remark}

\section{Eigenvectors of $\mathcal A$ and $\mathcal A^*$}
 \label{s:eigen}
 Our aim is to study the long-time behaviour  of the semigroup $\{T(t)\}_{t\ge 0}$. The strategy is the following. First we check  
that the adjoint operator $\mathcal A^*$ of $\mathcal A$ has a positive eigenvector $v=v(x_b,a)$ corresponding to some positive eigenvalue $\lambda$.
Then we introduce the semigroup $\{P(t)\}_{t\ge 0}$ given by $P(t)f=e^{-\lambda t}T(t)f$ defined on the space $E_1=L^1(X,\mathcal B(X),\mu)$
with the measure $\mu$ given by $d\mu=v\, d\ell$.
Then we prove that semigroup $\{P(t)\}_{t\ge 0}$ has an invariant density $f_i$  and $\lim_{t\to\infty} P(t)f=f_i$ 
for each density $f$. Finally, we translate this result in terms of the semigroup $\{T(t)\}_{t\ge 0}$.
 
We first study some properties of the adjoint operator of $\mathcal A$. Denote by $H$ the operator $H\colon E\to L^1[\dxb,\gxb]$
defined by $Hf(x_b)=\mathcal Pf(x_b)$.  
Then the operator 
$H^*\colon L^{\infty}[\dxb,\gxb] \to E^*$ is given by
\begin{equation}
\label{H*}
H^*f(x_b,a)=2q(x_b,a)P^*_af(x_b)=2q(x_b,a)f(S_a(x_b)).
\end{equation}
It should be noted that we omit in the last product the factor $\mathbf 1_{[x_a,\gxb]}(x_b)$ because $q(x_b,a)=0$ for $x_b\le x_a$.

\begin{lemma}
\label{eigA*}
Let 
\begin{equation*}
\mathcal D^{\odot}=\{f\in C(X),\,\,\, \frac{\partial f}{\partial a}\in C(X),\,\,\,f(x_b,\ga(x_b))=0,\,\,\,\frac{\partial f}{\partial a}(x_b,\ga(x_b))=0\}.
\end{equation*} 
Then $\mathcal D^\odot\subset \mathcal D(\mathcal A^*)$ and 
\begin{equation}
\label{eigenvec-A}
\mathcal A^*f=\frac{\partial f}{\partial a}+H^*{\tilde f}\quad\textrm{for $f\in \mathcal D^{\odot}$},   
\end{equation}
where $\tilde f(x_b)=f(x_b,0)$.
\end{lemma}
\begin{proof}
If $f\in \mathcal D^{\odot}$ and $\varphi \in \mathcal D(\mathcal A)$ then
\[
\begin{aligned}
\langle f,\mathcal A\varphi\rangle
&=-\int_{\dxb}^{\gxb}\int_0^{\ga(x_b)} f(x_b,a)\frac{\partial \varphi}{\partial a}(x_b,a)\,da\,dx_b\\
&=\int_{\dxb}^{\gxb} f(x_b,0)\varphi(x_b,0)\,dx_b+
\Big\langle \frac{\partial f}{\partial a},\varphi\Big\rangle
\\
&=\int_{\dxb}^{\gxb} f(x_b,0)(H\varphi)(x_b)\,dx_b+
\Big\langle \frac{\partial f}{\partial a},\varphi\Big\rangle
\\
&=\langle H^*\tilde f,\varphi\rangle+
\Big\langle \frac{\partial f}{\partial a},\varphi\Big\rangle= \Big\langle \frac{\partial f}{\partial a}+ H^*\tilde f,\varphi\Big\rangle.
\end{aligned}
\]
Thus $\mathcal D^{\odot}\subset \mathcal D(\mathcal A^*)$ and \eqref{eigenvec-A} holds.
\end{proof}

Now we will check that the operator $\mathcal A^*$ has a positive eigenvector $v\in \mathcal D^{\odot}$ for some eigenvalue $\lambda >0$. Let $\mathcal A^*v=\lambda v$. 
From (\ref{eigenvec-A}) it follows that
\begin{equation}
\label{H*2}
\lambda v -\frac{\partial v}{\partial a}=H^*{\tilde v}.
\end{equation}
The problem is that the adjoint semigroup $\{T^*(t)\}_{t\ge 0}$ is not continuous.
Instead of this semigroup we may use 
the sun dual semigroup, but it will be more convenient to consider a little modification of the sun dual semigroup.
Consider a semigroup $\{T^{\odot}(t)\}_{t\ge 0}$ 
defined on the space 
\[
\widetilde C(X)=\{f\in C(X)\colon\,\,\, f(x_b,\ga(x_b))=0\} 
\]
with the infinitesimal generator $\mathcal A^{\odot}$ 
with the domain 
$\mathcal D(\mathcal A^{\odot})=D^{\odot}$
and given by the same formula as $\mathcal A^*$.  
Let $\mathcal A^{\odot}_0f=\frac{\partial f}{\partial a}$ and  
$\mathcal D(\mathcal A^{\odot}_0)=\mathcal D(\mathcal A^{\odot})$.
Then $\mathcal A_0^{\odot}$ is 
the infinitesimal generator of a $C_0$-semigroup
$\{T^{\odot}_0(t)\}_{t\ge 0}$ 
on the space  $\widetilde C(X)$
given by the formula
\[
 T^{\odot}_0(t)f(x_b,a)=\begin{cases} 
f(x_b,a+t)&\textrm{for $a\le \ga(x_b)-t$},\\
0&\textrm{for $a> \ga(x_b)-t$}.
\end{cases}
\]
Let $H^{\odot}\colon C[\dxb,\gxb]\to \widetilde C(X)$ be given by
$H^{\odot}f(x_b,a)=2q(x_b,a)f(S_a(x_b))$.
Then  $\mathcal A^{\odot}f=\mathcal A^{\odot}_0f+H^{\odot}\tilde f$ for $f\in  \mathcal D(\mathcal A^{\odot})$.
Denote by $R(\lambda,\mathcal A^{\odot}_0)$ the resolvent of the operator $\mathcal A^{\odot}_0$.

\begin{lemma}
\label{K-lambda}
Let $K_{\lambda}\colon C[\dxb,\gxb]\to C[\dxb,\gxb]$, $\lambda\ge 0$, be the integral operator given by
\begin{equation}
\label{eigenvec-A4}
K_{\lambda}\tilde v(x_b)=\int_{\da(x_b)}^{\ga(x_b)}2e^{-\lambda a}q(x_b,a)\tilde v(S_a(x_b))\,da.
\end{equation}
If $\tilde v$ is a function such that $K_{\lambda}\tilde v=\tilde v$, then 
the function
\begin{equation}
\label{eigenvec-A2}
v(x_b,a)=\int_a^{\infty}H^{\odot}\tilde v(x_b,s)e^{-\lambda(s-a)}\,ds.
\end{equation}
satisfies {\rm(\ref{H*2})} and $v\in \mathcal D(\mathcal A^{\odot})$. 
\end{lemma}
\begin{proof}
If $v\in \widetilde C(X)$ satisfies the equation 
\begin{equation}
\label{H*3}
v=R(\lambda,\mathcal A^{\odot}_0)H^{\odot}{\tilde v}
\end{equation}
then $v$ also satisfies (\ref{H*2}).
Since $R(\lambda,\mathcal A^{\odot}_0)f=\int_0^{\infty} e^{-\lambda s}T^{\odot}_0(s)f\,ds$
we have 
\[
R(\lambda,\mathcal A^{\odot}_0)f(x_b,a)
=\int_a^{\infty}f(x_b,s)e^{-\lambda(s-a)}\,ds.
\]
Now (\ref{H*3}) can be written as the integral equation \eqref{eigenvec-A2}.
Observe that in order to find $v(x_b,a)$ it is enough to solve (\ref{eigenvec-A2}) for $a=0$.  
Equation (\ref{eigenvec-A2}) for $a=0$ takes the form 
\begin{equation*}
v(x_b,0)=\int_0^{\infty}H^{\odot}\tilde v(x_b,s)
e^{-\lambda s}\,ds.
\end{equation*}
In the above formula we replace $s$ by $a$ and apply (\ref{H*}). Then
\[
v(x_b,0)=\int_{\da(x_b)}^{\ga(x_b)}2e^{-\lambda a}q(x_b,a)v(S_a(x_b),0)\,da.
\]
Thus if $K_{\lambda}\tilde v=\tilde v$, then $v$ given by \eqref{eigenvec-A2}
satisfies \eqref{H*2}. Since $v$ belongs to the range of resolvent $R(\lambda,\mathcal A^{\odot}_0)$, 
we have    
 $v\in \mathcal D(\mathcal A^{\odot}_0)=\mathcal D(\mathcal A^{\odot})$. 
\end{proof}

We want to prove that there exists a constant $\lambda>0$ and a positive function  $\tilde v\in C[\dxb,\gxb]$  such that 
$K_{\lambda}\tilde v=\tilde v$. 
We split the proof of this fact into two lemmae.

\begin{lemma}
\label{lemma-eigenfunction1}
For each $\lambda\ge 0$ the spectral radius $r(K_{\lambda})$ of $K_{\lambda}$ 
is a positive, isolated and simple eigenvalue of $K_{\lambda}$ associated with a positive eigenfunction $\tilde v_{\lambda}\in C[\dxb,\gxb]$.
\end{lemma}
\begin{proof}
In order to check this property we write the operator $K_{\lambda}$ in the standard integral form.
We substitute $y(a)=S_a(x_b)$ in (\ref{eigenvec-A4}). Then $da=2\,dy/g(2y)$ and we find that the
operator $K_{\lambda}$ can be written in the form
\[
K_{\lambda}\tilde v(x_b)= \int_{\dxb}^{\gxb} k_{\lambda}(x_b,y)\tilde v(y)\,dy,\quad 
k_{\lambda}(x_b,y)=\frac{4e^{-\lambda a(y;x_b)}q(x_b,a(y;x_b))}{g(2y)},
\]
where
\[
a(y;x_b)=\int_{x_b}^{2y}\frac{dr}{g(r)}.
\]
The expression $a(y;x_b)$ has the following interpretation. If $x_b$ is the initial size of a mother cell and it splits at the age 
$a(y;x_b)$,  then $y$ is the initial size of its daughter cells.
Since the function $g$ is continuous and positive, the kernel $k_{\lambda}$  is continuous and nonnegative.
Moreover, $k_{\lambda}(x_b,x_b)>0$  for  all $x_b\in (\dxb,\gxb)$.
Indeed, $k_{\lambda}(x_b,x_b)>0$ if and only if $q(x_b,a(x_b;x_b))>0$.  The last inequality follows from (A6)
  (see Fig.~\ref{r:cell-cyc2}).
This implies that the
spectral radius $r(K_{\lambda})$ is a  positive, isolated and simple eigenvalue of $K_{\lambda}$ associated with an  eigenfunction $\tilde v_{\lambda}\in C[\dxb,\gxb]$
such that $\tilde v_{\lambda}(x_b)>0$ for $x_b\in (\dxb,\gxb)$ (see comments after the proof of Theorem 7.4 of \cite{AL}). 
Observe that $\tilde v_{\lambda}$ is also positive at $\dxb$ and $\gxb$, because $\tilde v_{\lambda}(y)>0$ for $y\in(\dxb,\gxb)$
and the functions $y\mapsto q(\dxb,a(y;\dxb))$ and $y\mapsto q(\gxb,a(y;\gxb))$ are positive on some nontrivial intervals. 
\end{proof}

\begin{lemma}
\label{lemma-eigenfunction2}
There exists $\lambda>0$  such that $r(K_{\lambda})=1$.
\end{lemma}
\begin{proof}
First we check that the function $\lambda\mapsto K_{\lambda}$ is 
continuous  with respect to the operator norm.
Indeed, let $\lambda_1\le \lambda_2$ and $T_{\lambda_1,\lambda_2}=K_{\lambda_1}-K_{\lambda_2}$.
Since $\int q(x_b,a)\,da=1$ and $e^{-\lambda_1 a}-e^{-\lambda_2 a}\le \ga(\lambda_2-\lambda_1)$, we have
\[
\|T_{\lambda_1,\lambda_2}\|\le \max\limits_{\dxb\le x_b\le\gxb} \int_{\da(x_b)}^{\ga(x_b)}2\big( e^{-\lambda_1 a}-e^{-\lambda_2 a}\big) q(x_b,a)\,da\le 
2\ga(\lambda_2-\lambda_1).
\] 
Since the function  $k_{\lambda}$ is continuous, the operator
$K_{\lambda}\colon C[\dxb,\gxb]\to C[\dxb,\gxb]$ is  compact for each $\lambda\ge 0$.
The spectral radius mapping restricted to
compact linear bounded operators on any Banach space
is continuous with respect to the
operator norm (see e.g. Theorem~2.1 of \cite{Degla}).
Thus the function
$\lambda\mapsto r(K_{\lambda})$ is continuous.
Observe that, $r(K_0)=2$, because $\|K_0\|=2$ and $K_0\mathbf 1_{[\dxb,\gxb]}=2\cdot \mathbf 1_{[\dxb,\gxb]}$.
Now, let $\bar \lambda>0$ be a constant such that $e^{-\bar\lambda \da}\le 1/4$. Then 
\[
K_{\bar\lambda} \tilde v(x_b)\le \frac12 \int_{\da(x_b)}^{\ga(x_b)}q(x_b,a)\tilde v(S_a(x_b))\,da\le \frac12\|\tilde v\|
\quad \textrm{for $\tilde v\ge 0$,}
\]
and consequently $r(K_{\bar\lambda}) \le \|K_{\bar\lambda}\|\le \frac12$. From the continuity of the function 
$\lambda\mapsto r(K_{\lambda})$ it follows that there exists a $\lambda\in (0,\bar\lambda)$ such that 
$r(K_{\lambda})=1$. 
\end{proof}
Now we apply formula (\ref{eigenvec-A2}) to find a nonnegative eigenfunction $v(x_b,a)$ of the operator $\mathcal A^*$.

\begin{proposition}
\label{prop-eig1}
The operator $\mathcal A^*$ has an eigenvalue $\lambda>0$ 
and a corresponding eigenfunction $v$ such that 
\begin{equation}
\label{eigenfun-A*}
c_1\Phi(x_b,a)\le  v(x_b,a)\le c_2 \Phi(x_b,a) 
\end{equation}
for some positive constants $c_1$ and $c_2$ independent of $x_b$ and $a$.
\end{proposition} 

\begin{proof}
According to Lemma~\ref{lemma-eigenfunction2} there exists $\lambda>0$  such that $r(K_{\lambda})=1$.
Let $\tilde v_{\lambda}$ be a positive fixed point of $K_{\lambda}$. Then from formulae (\ref{H*}) and (\ref{eigenvec-A2})
it follows that 
\[
v(x_b,a)=\int_a^{\infty}
2q(x_b,s)\tilde v_{\lambda}(S_s(x_b))e^{-\lambda(s-a)}\,ds
\]
is the eigenfunction of the operator $\mathcal A^*$ corresponding to $\lambda$. Since the functions $\tilde v_{\lambda}(S_s(x_b))$ 
and $e^{-\lambda(s-a)}$ are bounded above and bounded away from zero, 
there exist positive constants $c_1$ and $c_2$ such that
\[
c_1\int_a^{\infty}   q(x_b,s)\,ds  \le  v(x_b,a)\le c_2\int_a^{\infty}   q(x_b,s)\,ds 
\]
and (\ref{eigenfun-A*}) follows from the definition of $\Phi$.
\end{proof}

From now on $\lambda$ denotes  the eigenvalue from Proposition~\ref{prop-eig1}. 
\begin{lemma}
\label{l:opJ}
A function $f_i(x_b,a)$ is an eigenvector of $\mathcal A$ corresponding to $\lambda$
if and only if 
\begin{equation}
\label{A-Eigenv1}
f_i(x_b,a)=e^{-\lambda a}f_i(x_b,0)\quad\textrm{for}\quad a\le\ga(x_b)
\end{equation}
and the function $f(x_b)=f_i(x_b,0)$ satisfies the equation $Jf=f$, where the operator $J\colon L^1[\dxb,\gxb]\to L^1[\dxb,\gxb]$ is given by the formula
\begin{equation}
\label{J-eig1}
Jf(x_b)=\int_{\dxb}^{2x_b}2e^{-\lambda a(x_b;y)}q(y,a(x_b;y))f(y)\,dy.
\end{equation}
\end{lemma}
\begin{proof}
If a function $f_i(x_b,a)$ is an eigenvector of $\mathcal A$ corresponding to $\lambda$, then 
the function $z(t,x_b,a)=e^{\lambda t} f_i(x_b,a)$ is a solution of (\ref{eq1-n})--(\ref{eq3-n}).
Substituting   $z=e^{\lambda t} f_i$ into (\ref{eq1-n})--(\ref{eq2-n}) we obtain
\begin{equation}
\label{J-eig1a}
\lambda f_i(x_b,a)+\frac{\partial f_i}{\partial a}(x_b,a)=0,\quad
f_i(x_b,0)=\mathcal Pf_i(x_b).
\end{equation}
From the first of equations \eqref{J-eig1a} it follows that the function $(x_b,a)\mapsto e^{\lambda a} f_i(x_b,a)$ has 
zero partial derivative with respect to $a$.
Thus $e^{\lambda a} f_i(x_b,a)=f_i(x_b,0)$, where $f_i(x_b,0)$ is the value of the trace operator $\mathcal T$ on $f_i$.
Therefore  \eqref{A-Eigenv1} holds
and the function $x_b\mapsto  f_i(x_b,0)$ satisfies the following
integral equation
\begin{equation}
\label{A-Eigenv2}
f_i(x_b,0)=\int_0^{\infty}2e^{-\lambda a}\Big(P_a\big(q(\cdot,a)f_i(\cdot,0)\big)\Big)(x_b)\,da.
\end{equation}
Since
\[
\Big(P_a\big(q(\cdot,a)f(\cdot)\big)\Big)(x_b)
=\frac{d}{dx_b}\big(S_a^{-1}(x_b)\big)q\big(S_a^{-1}(x_b),a\big)f\big(S_a^{-1}(x_b)\big),
 \]
the substitution $y=S_a^{-1}(x_b)$ to \eqref{A-Eigenv2} gives \eqref{J-eig1}.
\end{proof}
From \eqref{J-eig1} it follows that $J$ is an integral operator with a continuous kernel.
In particular $Jf\in C[\dxb,\gxb]$ for $f\in L^1[\dxb,\gxb]$ and the operator 
$J$ restricted to $C[\dxb,\gxb]$ is a continuous and positive.

The operators $J$ and $K_{\lambda}$ are adjoint, i.e. 
\begin{equation*}
\int_{\dxb}^{\gxb} g(x_b)Jf(x_b)\,dx_b=\int_{\dxb}^{\gxb} K_{\lambda}g(x_b)f(x_b)\,dx_b
\end{equation*}
for $f\in L^1[\dxb,\gxb]$, $g\in L^{\infty}[\dxb,\gxb]$.

\begin{lemma}
\label{lemma-eigenfunction-J1}
There exists $\tilde f_i\in C[\dxb,\gxb]$ such that $J\tilde f_i= \tilde f_i$ and
$\tilde f_i(x_b)>0$ for $x_b\in (\dxb,\gxb)$. 
The function $\tilde f_i$ is 
the unique, up to a multiplicative constant, fixed point of  $J$.
\end{lemma}
\begin{proof}
Using the same arguments as in Lemma~\ref{lemma-eigenfunction1} we prove that 
there exists an eigenfunction $\tilde f_i\in C[\dxb,\gxb]$ of $J$ such that 
$\tilde f_i(x_b)>0$ for $x_b\in (\dxb,\gxb)$. This eigenfunction is indeed a fixed point of $J$  because 
$\langle J\tilde f_i,\tilde v_{\lambda}\rangle=\langle \tilde f_i,K_{\lambda}\tilde v_{\lambda}\rangle =\langle \tilde f_i,\tilde v_{\lambda}\rangle$.
Since $r=1$ is an isolated and simple eigenvalue of $J$, the function $\tilde f_i$ is 
the unique, up to a multiplicative constant, fixed point of~$J$.
\end{proof}

\begin{remark}
\label{positivivty-boundary}
It is  generally not true  that  $\tilde f_i(\dxb)>0$ and $\tilde f_i(\gxb)>0$.
If we assume additionally that $S_{\da(x_b)}(x_b)=\dxb$ for $x_b\in [\dxb,\dxb+\delta]$, $\delta>0$, then 
$\tilde f_i(\dxb)>0$ because a mother cell with the initial size $x_b\in [\dxb,\dxb+\delta]$ can have a daughter cell with the initial size $\dxb$.
Analogously,  if $S_{\ga(x_b)}(x_b)=\gxb$ for 
$x_b\in [\gxb-\delta,\gxb]$, $\delta>0$, then $\tilde f_i(\gxb)>0$.
\end{remark}

From Lemma~\ref{lemma-eigenfunction-J1} and from formulae (\ref{A-Eigenv1}) and (\ref{A-Eigenv2}) we have 
\begin{proposition}
\label{prop-eig2}
Let $\tilde f_i$ be the function from Lemma~$\ref{lemma-eigenfunction-J1}$. If $f_i(x_b,a)=e^{-\lambda a}\tilde f_i(x_b)$, then
$\mathcal Af_i=\lambda f_i$. The function $f_i$ is 
the unique, up to a multiplicative constant, eigenfunction of  $\mathcal A$ corresponding to the eigenvalue $\lambda$.

\end{proposition}

\section{Asymptotic behaviour}
 \label{s:asyp-beh}
We precede the formulation of the main result of this section by some definitions and some general theorem concerning asymptotic stability of stochastic semigroups.

Let a triple $(X,\Sigma,m)$ be a $\sigma$-finite measure space.
Denote by $D$ the subset of the space
$L^1=L^1(X,\Sigma,m)$
which contains all
densities
\[
D=\{f\in \, L^1\colon  \,f\ge 0,\,\, \|f\|=1\}.
\]
A $C_0$-semigroup $\{P(t)\}_{t\ge 0}$ of linear operators on $L^1$
is called \textit{stochastic semigroup} or \textit{Markov semigroup}
if   $P(t)(D)\subseteq  D$ for each $t\ge 0$.

A stochastic semigroup $\{P(t)\}_{t\ge 0}$ is
{\it asymptotically stable} if  there exists a density $f_i$   such that
\begin{equation}
\label{d:as}
\lim _{t\to\infty}\|P(t)f-f_i\|=0 \quad \text{for}\quad f\in D.
\end{equation}
 From (\ref{d:as}) it follows immediately that  $f_i$ is {\it invariant\,} with respect to
$\{P(t)\}_{t\ge 0}$, i.e.  $P(t)f_i=f_i$ for
each $t\ge 0$.

A stochastic semigroup $\{P(t)\}_{t\ge 0}$
is called {\it partially integral} if there exists a measurable
function $k\colon (0,\infty)\times X\times X\to[0,\infty)$, called a
{\it kernel}, such that
\begin{equation*}
P(t)f(x)\ge\int_X k(t,x,y)f(y)\,m (dy)
\end{equation*}
for every density $f$ and
\begin{equation*}
\int_X\int_X  k(t,x,y)\,m(dx)\,m(dy)>0
\end{equation*}
for some $t>0$. The following result was proved in \cite{PR-jmaa2}.

\begin{theorem}
\label{asym-th2}
Let $\{P(t)\}_{t\ge 0}$ be a partially integral stochastic
semigroup. Assume that the  semigroup $\{P(t)\}_{t\ge 0}$ has
a unique invariant density $f_i$. If $f_i>0$ a.e., then the semigroup
$\{P(t)\}_{t\ge 0}$ is asymptotically stable.
\end{theorem}

New results concerning positive operators on Banach lattices similar in spirit to Theorem~\ref{asym-th2} may be found in \cite{Gerlach-Gluck1,Martin-Gluck2}.

Investigation of the long-time behaviour  of the semigroup $\{T(t)\}_{t\ge 0}$ 
can be reduced to the study of asymptotic stability of some stochastic semigroup. 
Let $\lambda$ and $v$ be the eigenvalue and the eigenfunction from Proposition~\ref{prop-eig1}.
We define a semigroup  $\{P(t)\}_{t\ge 0}$ as the extension of semigroup  $\{e^{-\lambda t}T(t)\}_{t\ge 0}$  to 
the space $E_1=L^1(X,\mathcal B(X),\mu)$ with measure $\mu$ given by $d\mu=v\, d\ell$.
Observe that we can indeed  extend the semigroup  $\{e^{-\lambda t}T(t)\}_{t\ge 0}$ to a stochastic semigroup on $E_1$. 
Since $\mathcal A^*v=\lambda v$, we have $T^*(t)v=e^{\lambda t} v$. 
If $f\in E$ then $P(t)f=e^{-\lambda t}T(t)f$ and
\begin{align*}
&\iint\limits_X P(t)f(x_b,a)\,\mu(dx_b,da)=\iint\limits_X e^{-\lambda t}T(t)f(x_b,a) v(x_b,a)\,dx_b\,da\\
&=\iint\limits_X f(x_b,a) e^{-\lambda t}T^*(t)v(x_b,a)\,dx_b\,da
=\iint\limits_X f(x_b,a) v(x_b,a)\,dx_b\,da\\
&=\iint\limits_X f(x_b,a) \,\mu(dx_b,da).
\end{align*}
Since the function $v$ is bounded and positive almost everywhere, $E$ is dense in $E_1$. 
If $f\in E_1$, we choose a sequence $(f_n)$ from $E$ such that $f_n\to f$ in $E_1$ and define $P(t)f=\lim\limits_{n\to\infty} P(t)f_n$ in $E_1$.
Since the operators $P(t)$ are positive and preserve the integral with respect to $\mu$, this extension is uniquely defined  and  
$\{P(t)\}_{t\ge 0}$ is a stochastic semigroup on $E_1$.

In order to prove asymptotic stability of the semigroup $\{P(t)\}_{t\ge 0}$ we need to add an additional assumption concerning function $g$:
  
\noindent  (A7) there exists $x\in (\dxb,\gxb)$ such that $g(2x)\ne 2g(x)$.

We precede the formulation of a theorem on asymptotic stability of $\{P(t)\}_{t\ge 0}$ by the following lemma.
\begin{lemma}
\label{lemma-partially-integral}
Assume {\rm (A1)--(A7)}.
Then the semigroup $\{T(t)\}_{t\ge 0}$ is partially integral.
\end{lemma}
\begin{proof}
Observe that the operator $T(t)$ has the kernel $k(t,x,y)$ if and only if the operator $T^*(t)$ has the kernel $k^*(t,x,y)=k(t,y,x)$. Thus, in order to prove that 
the semigroup $\{T(t)\}_{t\ge 0}$ is partially integral it is sufficient to check that the semigroup $\{T^{\odot}(t)\}_{t\ge 0}$ has the similar property.
The semigroup $\{T^{\odot}(t)\}_{t\ge 0}$ is given by the \textit{Dyson-Phillips expansion}
\begin{equation}
\label{eq:dpf1b}
T^{\odot}(t)f=\sum_{n=0}^\infty  T^{\odot}_n(t)f, 
\end{equation}
where
\begin{equation*}
T^{\odot}_{n+1}f(t)=
\int_0^tT^{\odot}_{0}(\tau)\mathcal HT^{\odot}_n(t-\tau)f\,d\tau, \quad
n\ge 0,
\end{equation*}
where $\mathcal Hf(x_b,a)=2q(x_b,a)f(S_a(x_b),0)$
and $T^{\odot}_0(t)f(x_b,a)=f(x_b,a+t)$
for $a\le \ga(x_b)-t$.
Since $\mathcal HT^{\odot}_0(t-\tau)f(x_b,a)= 2q(x_b,a)f(S_a(x_b),t-\tau)$,
we have 
\begin{align*}
T^{\odot}_1f(t)(x_b,a)&
=\int_0^tT^{\odot}_{0}(\tau)\mathcal HT^{\odot}_0(t-\tau)f(x_b,a)\,d\tau\\
&=\int_0^t 2q(x_b,a+\tau)f(S_{a+\tau}(x_b),t-\tau)
\,d\tau.
\end{align*}
Analogously, since 
\[ 
T^{\odot}_1(t-\tau_1)f(x_b,a)=\int_0^{t-\tau_1} 2q(x_b,a+\tau)f(S_{a+\tau}(x_b),t-\tau_1-\tau)
\,d\tau,
\]
we have  
\[ 
\mathcal HT^{\odot}_1(t-\tau_1)f(x_b,a)=2q(x_b,a)\int_0^{t-\tau_1} 2q(S_a(x_b),\tau)f(S_{\tau}(S_a(x_b)),t-\tau_1-\tau)
\,d\tau,
\]
and finally 
 \begin{align*}
T^{\odot}_2f(t)(x_b,a)&
=\int_0^tT^{\odot}_{1}(\tau_1)\mathcal HT^{\odot}_1(t-\tau_1)f(x_b,a)\,d\tau_1\\
&=\int_0^t
2q(x_b,a+\tau_1)\int_0^{t-\tau_1} 2q(S_{a+\tau_1}(x_b),\tau)\\
&\hspace{3cm}{}\cdot f(S_{\tau}(S_{a+\tau_1}(x_b)),t-\tau_1-\tau)
\,d\tau \,d\tau_1.
\end{align*}
We substitute in the last integral $\tilde x=S_{\tau}(S_{a+\tau_1}(x_b))$
and $\tilde a=t-\tau_1-\tau$. Then
\begin{align*}
\frac{\partial \tilde x}{\partial \tau}&=\frac12g\big(S_{\tau}(S_{a+\tau_1}(x_b))\big), 
\\ 
\frac{\partial \tilde x}{\partial \tau_1}&=\frac12
\frac{g(\pi_{\tau}S_{a+\tau_1}(x_b))}{g(S_{a+\tau_1}(x_b))}
\cdot \frac12 g\big(S_{a+\tau_1}(x_b)\big)=
\frac14g(\pi_{\tau}S_{a+\tau_1}(x_b)),\\
\frac{\partial \tilde a}{\partial \tau}&=
\frac{\partial \tilde a}{\partial \tau_1}=-1. 
\end{align*} 
Let $\mathcal J_{\tau,\tau_1}$ be the Jacobian matrix of the transformation $(x_b,a)\mapsto (\tilde x,\tilde a)$. Then
\[
\det \mathcal J_{\tau,\tau_1}(x_b,a)=\frac14g(\pi_{\tau}S_{a+\tau_1}(x_b))-\frac12g\big(S_{\tau}(S_{a+\tau_1}(x_b))\big).
\]
According to (A7) there exists
$x\in (\dxb,\gxb)$ such that $g(2x)\ne 2g(x)$.
We find $x^1_b \in (\dxb,\gxb)$ and $\tau^0>0$ such that 
$q(x_b^1,\tau^0)>0$ and $S_{\tau^0}(x_b^1)=x$, i.e. 
 $x$ and $x_b^1$ are the initial sizes of  daughter and mother cells.
 Next we find $x^0_b \in (\dxb,\gxb)$, $a^0>0$, and $\tau_1^0>0$ such that 
$q(x_b^0,a^0+\tau_1^0)>0$ and $S_{a^0+\tau_1^0}(x_b^0)=x^1_b$. We also choose $t>0$ such that the point $(x,t-\tau^0-\tau_1^0)$ lies in the interior of the set $X$.
Then we find a neighbourhood  $\mathcal U$ of the point $(\tau^0,\tau_1^0,x_b^0,a^0)$ such that 
$\det \mathcal J_{\tau,\tau_1}(x_b,a)\ne 0$, 
$q(x_b,a+\tau_1)>0$, $q(S_{a+\tau_1}(x_b),\tau)>0$,
$(S_{\tau}(S_{a+\tau_1}(x_b)),t-\tau-\tau_1)\in X$
for $(\tau,\tau_1,x_b,a)\in \mathcal U$. Thus there exist  neighbourhoods $V_1$ and $V_2$ of the points $(x_b^0,a^0)$ and $(x,t-\tau_1^0-\tau^0)$
and there exist $\varepsilon>0$ and a nonnegative kernel $k(t,x_b,a,\tilde x,\tilde a)$ 
such that $k(t,x_b,a,\tilde x,\tilde a)\ge \varepsilon$ for $(x_b,a,\tilde x,\tilde a)\in V_1\times V_2$ and
\[
T^{\odot}_2f(t)(x_b,a)\ge \iint\limits_{X} k(t,x_b,a,\tilde x,\tilde a)f(\tilde x,\tilde a) \,d\tilde x\,d\tilde a.
\] 
From (\ref{eq:dpf1b}) it follows that the semigroups $\{T^{\odot}(t)\}_{t\ge 0}$ and $\{T(t)\}_{t\ge 0}$ are partially integral.
\end{proof}

\begin{theorem}
\label{asym-as-2}
Assume {\rm (A1)--(A7)}.
Then the semigroup $\{P(t)\}_{t\ge 0}$ is asymptotically stable. 
The eigenfunction of $\mathcal A$ from  Proposition~$\ref{prop-eig2}$ is the invariant density of  $\{P(t)\}_{t\ge 0}$.
\end{theorem}
\begin{proof}
We need to check that the semigroup $\{P(t)\}_{t\ge 0}$ satisfies assumptions of Theorem~\ref{asym-th2}. 
Since the semigroup $\{T(t)\}_{t\ge 0}$ is partially integral and $P(t)f=e^{-\lambda t}T(t)f$ for $f\in E$,   
the same property has the semigroup $\{P(t)\}_{t\ge 0}$.
According to Proposition~$\ref{prop-eig2}$ there exists a function $f_i$ such that $\mathcal Af_i=\lambda f_i$ and $f_i>0$ a.e.
As $\mu(X)<\infty$ and $f_i$ is bounded,  $f_i$ is integrable with respect to $\mu$. 
Since  the eigenfunction is determined up to a multiplicative constant, we may assume that $\int_X f_i\,d\mu=1$.
Also according to Proposition~$\ref{prop-eig2}$
the function $f_i$ is the unique invariant density of $\{P(t)\}_{t\ge 0}$.
\end{proof}
\begin{theorem} 
\label{th:long-time-u}
For every $u_0\in E$ we have
\begin{equation}
\label{d:as4}
\lim _{t\to\infty}e^{-\lambda t}U(t)u_0=\Phi f_i\iint\limits_X u_0(x_b,a)\Psi(x_b,a)v(x_b,a)\,dx_b\,da \quad \textrm{in $E$}.
\end{equation}
Moreover, $\Phi f_i$ and $\Psi v$ are eigenfunctions of the semigroups $\{U(t)\}_{t\ge 0}$ and $\{U^*(t)\}_{t\ge 0}$ corresponding to the eigenvalue $\lambda$.
\end{theorem}
\begin{proof}
Condition of asymptotic stability of the semigroup $\{P(t)\}_{t\ge 0}$ can be written in the following equivalent form: 
for every $f\in E_1$ we have
\begin{equation}
\label{d:as1}
\lim _{t\to\infty}P(t)f=f_i\iint\limits_X f(x_b,a)v(x_b,a)\,dx_b\,da \quad \textrm{in $E_1$}.
\end{equation}
We can extend the semigroup $\{T(t)\}_{t\ge 0}$ to a $C_0$-semigroup on $E_1$ setting $T(t)f=e^{\lambda t}P(t)f$ for $f\in E_1$.
From (\ref{d:as1}) it follows that 
\begin{equation}
\label{d:as2}
\lim _{t\to\infty}e^{-\lambda t}T(t)f=f_i\iint\limits_X f(x_b,a)v(x_b,a)\,dx_b\,da \quad \textrm{in $E_1$}.
\end{equation}
Now we return to the problem  (\ref{eq1})--(\ref{eq3}). We recall that 
after substitution $z(t,x_b,a)=u(t,x_b,a)\Psi(x_b,a)$  and  $z_0(x_b,a)=u_0(x_b,a)\Psi(x_b,a)$ we have replaced this 
problem by  the system (\ref{eq1-n})--(\ref{eq3-n}) and the semigroup $\{T(t)\}_{t\ge 0}$ describes the evolution of the solutions of 
this system. Since $z=u\Psi$ and $z_0=u_0\Psi$, we have $u(t)=\Phi T(t)(u_0\Psi)$ because $\Phi=1/\Psi$.
Thus we can define a semigroup $\{U(t)\}_{t\ge 0}$ corresponding to (\ref{eq1})--(\ref{eq3}) by 
\begin{equation}
\label{d:as-defU}
U(t)u_0=\Phi T(t)(u_0\Psi).
\end{equation}
From inequalities  (\ref{eigenfun-A*}) it follows that $u_0\Psi\in E_1$ if and only if $u_0\in E$ and 
$\{U(t)\}_{t\ge 0}$ is a $C_0$-semigroup on the space $E$.
It should be noted that we consider solutions of (\ref{eq1})--(\ref{eq3}) for a wider class of initial conditions
because we do not assume that $u_0$ satisfies inequality (\ref{def-u_0}). From (\ref{d:as2}) it follows that   
\begin{equation*}
\lim _{t\to\infty}e^{-\lambda t}\Psi U(t)u_0=f_i\iint\limits_X u_0(x_b,a)\Psi(x_b,a)v(x_b,a)\,dx_b\,da \quad \textrm{in $E_1$}.
\end{equation*}
Using again inequalities  (\ref{eigenfun-A*}) we finally obtain \eqref{d:as4}.
\end{proof}

Property (\ref{d:as4})  is called the asynchronous exponential growth of the semigroup  $\{U(t)\}_{t\ge 0}$.
Precisely, we say that a semigroup $\{U(t)\}_{t\ge 0}$ on a Banach space $\mathbb X$
has \textit{asynchronous} (or \textit{balanced}) \textit{exponential growth} if there exist $\lambda\in\mathbb C$,
a  nonzero $x_i\in \mathbb X$, and a nonzero linear functional $\alpha\colon \mathbb X\to \mathbb C$ 
such that 
\begin{equation*}
\lim_{t\to\infty}e^{-\lambda t}U(t)x=x_i\alpha(x)\quad\textrm{for $x\in\mathbb X$}.
\end{equation*}
It should be mentioned that one can find in literature, e.g. \cite{Webb-aeg}, a more general definition of asynchronous exponential growth,
where it is only assumed that $e^{-\lambda t}U(t)x$ converges to a nonzero finite rank operator.

\section{Remarks}
 \label{s:remarks}
\subsection{Chemostat}
 \label{ss:chemostat}
 In Section~\ref{s:model} we have mentioned that if we consider experiment in a chemostat, then we need to add to the right-hand side of equation (\ref{eq1}) the term $-Du(t,x_b,a)$. In this case we substitute $u(t,x_b,a)=e^{-Dt}\bar u(t,x_b,a)$ and then we check that the function $\bar u$ satisfies the system (\ref{eq1})--(\ref{eq3}). 
 From Theorem~\ref{th:long-time-u} we deduce that
\begin{equation}
\label{d:as5}
\lim _{t\to\infty}e^{(D-\lambda)t}U(t)u_0=\Phi f_i\iint\limits_X u_0(x_b,a)\Psi(x_b,a)v(x_b,a)\,dx_b\,da \quad \textrm{in $E$}.
\end{equation}
From~(\ref{d:as5}) it follows that in order to grow cells under constant environmental conditions, 
cells should be removed from the system with rate $D=\lambda$.

\subsection{Age-size structured model}
 \label{ss:age-size}
Now we consider an age-size structured model consistent with our biological description.  
Let $\bar p(x,a)\Delta t$  be the probability that a cell with size $x$ and age $a$ splits in the time interval of the length $\Delta t$.  
Since such a cell had the initial size $x_b=\pi_{-a}x$,
we see that
\begin{equation*}
\bar p(x,a)=p(\pi_{-a}x,a)= q(\pi_{-a}x,a)\big/\textstyle{\int_a^{\infty}} q(\pi_{-a}x,r)\,dr 
\end{equation*}
for $a<\ga(\pi_{-a}x)$. We set $\bar p(x,a)=0$ for $a\ge \ga(\pi_{-a}x)$.
Let  $w(t,x,a)$ be the number of cells   having  size $x$ and age $a$ at time $t$.
Then the function $w$ satisfies the following initial-boundary problem:
\begin{align}
\label{eq1-w}
&\frac{\partial w}{\partial t}(t,x,a)
 +\frac{\partial w}{\partial a}(t,x,a)
  +\frac{\partial (gw)}{\partial x}(t,x,a)
 =-\bar p(x,a)w(t,x,a),\\
&w(t,x,0)=4\int_{0}^{\infty}\bar p(2x,a)w(t,2x,a)\,da, 
\label{eq2-w}\\
&w(0,x,a)=w_0(x,a). 
\label{eq3-w}
\end{align}
We have the following relationship
between solutions of the systems (\ref{eq1})--(\ref{eq3}) and (\ref{eq1-w})--(\ref{eq3-w}):
\begin{equation}
\label{relacja-u-w}
\int_0^a\int_0^x u(t,x_b,r)\,dx_b\,dr=\int_0^a\int_0^{\pi_rx} w(t,y,r)\,dy\,dr.
\end{equation}
Differentiating both sides of (\ref{relacja-u-w}) with respect to $a$ and $x$ we obtain  
\[
u(t,x,a)=\frac{\partial (\pi_ax)}{\partial x} w(t,\pi_a x,a)
 =\frac{g(\pi_ax)}{g(x)} w(t,\pi_a x,a).
\]
Using the above formula and Theorem~\ref{th:long-time-u} we get
\[
e^{-\lambda t}\frac{g(\pi_ax)}{g(x)} w(t,\pi_a x,a)\to (\Phi f_i)(x,a) \iint\limits_X w_0(\pi_ax_b,a)h(x_b,a)\,dx_b\,da
\]
in $E$ as $t\to\infty$,
where  $h(x_b,a)=\Psi(x_b,a)v(x_b,a)g(\pi_ax_b)/g(x_b)$. Finally we conclude that   
\[
e^{-\lambda t} w(t,x,a)\to h_i(x,a) \iint\limits_X w_0(\pi_ax_b,a)h(x_b,a)\,dx_b\,da
\]
in $L^1$, where $h_i(x,a)=(\Phi f_i)(\pi_{-a}x,a)g(\pi_{-a}x)/g(x)$.

\subsection{Solutions with values in the space of measures}
 \label{ss:slution-in-measures}
If we study the dynamics of population growth of microorganisms starting from a single cell, then initial distribution of the population is described by a singular measure, precisely with a delta Dirac measure. Thus it is natural to consider a model which describes the evolution of measures instead of $L^1$ functions. 
We can introduce such a model by considering weak solutions.
Let $\{U(t)\}_{t\ge 0}$ be the semigroup introduced in Section~\ref{s:asyp-beh} and let  
$\{U^{\odot}(t)\}_{t\ge 0}$  be the ``dual semigroup" given by $U^{\odot}(t)u_0=\Psi T^{\odot}(t)(\Phi u_0)$ (see formula (\ref{d:as-defU})). Denote by  $\mathcal M(X)$ the space of all finite Borel measures on $X$.
For any measure $\nu_0\in \mathcal M(X)$ 
we define the \textit{weak solution} of the problem (\ref{eq1})--(\ref{eq3}) as a function $u\colon [0,\infty)\to \mathcal M(X)$, $u(t)=\nu_t$, 
where the measures $\nu_t$   
satisfy the condition  
\begin{equation}
\label{measure-sol-def}
\iint\limits_X f(x_b,a)\,\nu_t(dx_b,da)=\iint\limits_X U^{\odot}(t)f(x_b,a)\,\nu_0(dx_b,da)
 \end{equation}
 for all  $f\in C(X)$.  
Since the set $X$ is compact, the existence and uniqueness of the measures $\nu_t$ is a simple consequence of the Riesz representation theorem.

One can ask  about the long-time behaviour of the measures $\nu_t$. We are interested in convergence of measures in the total variation norm. We denote by $d(\nu,\bar\nu)_{TV}$  the distance between $\nu$ and $\bar\nu$
in the \textit{total variation norm} in  $\mathcal M(X)$. We recall that  
\[
d(\nu,\bar\nu)_{TV}=(\nu-\bar\nu)^+(X)+(\nu-\bar\nu)^-(X),
\]
where the symbols $\nu^+$ and $\nu^-$ denote the  positive and negative part of a signed measure $\nu$.
We can formulate Theorem~\ref{th:long-time-u} in a slightly stronger form: 
\begin{proposition} 
\label{prop:long-time-u-tv}
Assume that conditions {\rm (A1)--(A7)} hold.
Let $\nu_0\in \mathcal M(X)$ and let $\nu_{\infty}$ be the measure given by  
\[
\nu_{\infty}(A)=\int\limits_A \Phi(x_b,a) f_i(x_b,a) \,dx_b\,da\cdot \iint\limits_X\Psi(x_b,a)v(x_b,a)\,\nu_0(dx_b,\,da)
\]
for $A\in\mathcal B(X)$.
Then 
\begin{equation}
\label{measure-sol-converg}
\lim_{t\to\infty}d(e^{-\lambda t}\nu_t,\nu_{\infty})_{TV}=0.
\end{equation}
\end{proposition}
We only give some idea of the proof of Proposition~\ref{prop:long-time-u-tv}. 
We consider weak solutions connected with the semigroup $\{P^{\odot}(t)\}_{t\ge 0}$,
i.e. we replace in (\ref{measure-sol-def}) the semigroup  
 $\{U^{\odot}(t)\}_{t\ge 0}$ by $\{P^{\odot}(t)\}_{t\ge 0}$, where $P^{\odot}(t)=e^{-\lambda t}T^{\odot}(t)$.
It is enough to check that if $\nu_0$ is a probability measure then 
$\lim_{t\to\infty}d(\nu_t,\mu_*)_{TV}=0$, where $d\mu_*=f_id\mu$.
We write $\nu_t$ as a sum $\nu_t^a+\nu_t^s$, where $\nu_t^a$ is the absolutely continuous part of $\nu_t$ with respect to the Lebesgue measure 
and $\nu_t^s$ is the singular part of $\nu_t$. 
We deduce from conditions (A6) and  (A7) that there exist $t_0>0$ and $\varepsilon>0$ independent of $\nu_0$ 
such that $\nu^s_{t_0}(X)\le 1-\varepsilon$. The proof of this part is very technical but it uses similar arguments as the proof of 
Lemma~\ref{lemma-partially-integral}. From the last inequality it follows that 
 $\nu^s_{t}(X)\le (1-\varepsilon)^n$ for $t\ge nt_0$. Fix $\eta>0$ and let $t_1>0$ be such that $\nu^s_{t_1}(X)\le\eta$.
Let $f=d\nu_{t_1}^a/d\mu$. Then $\lim_{t\to\infty}\|P(t)f-f_i\int f\,d\mu\|_{E_1}=0$. Since 
$\int f\,d\mu\ge 1-\eta$ and $\nu^s_{t_1}(X)\le\eta$, we have 
$d(\nu_t,\mu_*)_{TV}\le 2\eta$.
As $\eta>0$ can be chosen arbitrary small we finally obtain   
$\lim_{t\to\infty}d(\nu_t,\mu_*)_{TV}=0$.

It should be noted that a similar result 
can be obtained by using the theory of positive recurrent and aperiodic Harris processes (see Theorem 13.3.3 of~\cite{Meyn-Tweedie}),
but to apply this theorem we need to formulate the problem properly in the language of stochastic processes. First, we construct a family of Markov processes corresponding to the semigroup $\{P(t)\}_{t\ge 0}$. In particular, we need to define additionally the processes started from points $(x_b,\ga(x_b))$.
Since Theorem 13.3.3 applies to discrete-time processes, we consider this Markov family for times $t=0,1,2,\dots$ and check assumptions of this theorem.
As a result we obtain that $\lim_{n\to\infty}d(\nu_n,\mu_*)_{TV}=0$. Finally, we pass from discrete time convergence to continuous time convergence.

\subsection{Case  $g(2x)=2g(x)$}
 \label{ss:g(2x)=2g(x)}
 A function $g$ satisfying condition $g(2x)=2g(x)$ for all $x\in [\dxb,\gxb]$ can be constructed in the following way.
Let $g\colon [\dxb,2\dxb]\to (0,\infty)$ be a given $C^1$-function such that $g(2\dxb)=2g(\dxb)$ and $g'(2\dxb)=g'(\dxb)$.
Then we define $g(x)=2^ng(2^{-n}x)$ for $ x\in [2^n \dxb,2^{n+1}\dxb]$. 
 
Observe that if $g(2x)=2g(x)$ for all $x\in [\dxb,\gxb]$, then the semigroup 
$\{U(t)\}_{t\ge 0}$ has no asynchronous exponential growth.
Indeed, consider a cell with initial size $x_b$.
Fix time $t>0$ and assume that the cell splits at age $a\le t$. 
Then the daughter cells at time $t$ have size 
\[
x(a)=\pi_{t-a}(\tfrac 12\pi_ax_b).
\]
Since
\[
x'(a)=
-g(\pi_{t-a}(\tfrac 12\pi_ax_b))
+
\frac{g(\pi_{t-a}(\tfrac 12\pi_ax_b))}
{g(\tfrac 12\pi_ax_b)}
\cdot \frac12g(\pi_ax_b)=0,
\]
the function $x$ is constant and $x(a)=\tfrac 12\pi_tx_b$. Thus the size of all daughter cells is exactly twice smaller than the size of the mother cell.
If $x_n(t)$ is the size of a cell from the $n$th generation then its mother, grandmother, etc. cells have sizes  $2x_n(t)$, $4x_n(t)$, $\dots$ 
But since cells have minimum and maximum size $\dxb$ and $2\gxb$, the maximum number of existing generations at a given time $t$ 
is not greater than $2+\log_2(\gxb/\dxb)$. Moreover, if $x_1\in (\dxb,2\dxb)$ and 
$f_1(x_b,a)=\mathbf 1_{(\dxb,x_1)}(x_b)$,  $f_2(x_b,a)=\mathbf 1_{(x_1,2\dxb)}(x_b)$, then  $U(t)f_1\cdot U(t)f_2\equiv 0$ for all $t\ge 0$. Consequently, 
the semigroup $\{U(t)\}_{t\ge 0}$ has no asynchronous exponential growth.

\begin{remark}
\label{r:irred-overlaping}
One can check that the semigroup $\{U(t)\}_{t\ge0}$ is \textit{irreducible} i.e. $\int_0^{\infty}U(t)f\,dt>0$ a.e. even if (A7) does not hold.
Our example shows that a semigroup can be irreducible but not overlapping supports.
A stochastic semigroup $\{P(t)\}_{t\ge0}$  is called \textit{overlapping supports} if 
$P(t)f_1\cdot P(t)f_2\ne 0$ for any two densities $f_1$ and $f_2$ and some $t=t(f_1,f_2)$.
Another simple example of irreducible stochastic semigroup which does not overlap supports
is the rotation semigroup. If  $X=S^1$ is a unit circle on the complex plain with centre 
$z_0=0$, $\,\Sigma=\mathcal B(X)$ is the 
$\sigma$-algebra of Borel subsets of $X$ and 
$m$ is the arc-Lebesgue measure on $X$, 
the rotation semigroup $\{P(t)\}_{t\ge 0}$
is given by $P(t)f(z)=f(ze^{it})$. 
\end{remark}

Now we consider a special case when $g(x)=\kappa x$, $\kappa>0$. We start at time $t=0$ with a single cell with size $x$.
Cells from the $n$th generation have  size $2^{-n}e^{\kappa t}x$ at time $t$. 
Then $\bar p(2^{-n}e^{\kappa t}x,a)\Delta t$ is  the probability that a cell 
from the $n$th generation with age $a$ splits in the time interval of the length $\Delta t$.
This observation allows us to describe the evolution of the population 
using discrete parameters. 
Denote by $w_n(t,a)$ the number of cells from the $n$th generation with age $a$  at time $t$.
Then the functions $w_n$ satisfy the following  infinite system of partial differential equations with boundary conditions:
\begin{align*}
&\frac{\partial w_n}{\partial t}(t,a)
 +\frac{\partial w_n}{\partial a}(t,a)
  =-\bar p\big(2^{-n}e^{\kappa t}x,a\big)w_n(t,a),\\
&w_n(t,0)=2\int_{0}^{\infty}\bar p\big(2^{1-n}e^{\kappa t}x,a\big)w_{n-1}(t,a)\,da. 
\end{align*}

 It should be noted that it is not easy to find a direct formula for the eigenvector  $f_i(x_b,a)$ of the operator $\mathcal A$ even in the case $g(x)=\kappa x$.
Indeed, we have $\pi_{-a}(2x_b)=2e^{-\kappa a}x_b$ and
\begin{equation*}
\begin{aligned}
P_af(x_b)&=\frac{2g(\pi_{-a}(2x_b))}{g(2x_b)}f(\pi_{-a}(2x_b))=\frac{2\pi_{-a}(2x_b)}{2x_b}f(\pi_{-a}(2x_b))\\
&=2e^{-\kappa a}f(2e^{-\kappa a}x_b).
\end{aligned}
 \end{equation*}
Then $f_i(x_b,a)=e^{-\lambda a}f_i(x_b,0)$, where $\lambda$ and $f_i(x_b,0)$ should be found by solving the following equation
\begin{equation*}
f_i(x_b,0)=\int_0^{\infty}4e^{-(\lambda+\kappa)a}
q(2e^{-\kappa a}x_b,a)f_i(2e^{-\kappa a}x_b,0)\,da,
\end{equation*}
which is not a simple task.

\section{Comparison with experimental data and other models}  
 \label{s:other2}
Modern experimental techniques enable studies of individual
cells growth in well-controlled environments. Especially interesting are experimental results concerning rod-shaped bacteria,
 for example \textit{E. coli}, \textit{C. crescentus} and \textit{B. subtilis} \cite{CSK,I-B,T-A,Wang},
because they change only their length. Although such bacteria have similar shape there are 
variety of distinct models of cell cycle and cell division. For example, we  consider  models with symmetric or asymmetric divisions,
with different velocities of proliferation, deterministic or stochastic growth of individuals or
models based on special assumptions as fixed cell length extension or models with target size division. 
  We give a short review of such models and show how to incorporate them to our model. 

\subsection{Models with exponential growth}
\label{ss:exponential growth}
Since experimental data suggest that cells grow exponentially, one can find a number of models with the assumption  $g(x)=\kappa x$
 but with various descriptions of the cell cycle length.

In \cite{CSK,GH,T-A,VK} it is considered an \textit{additive model} (or a \textit{constant $\Delta$ model}), where it is assumed that the difference 
$\Delta(x_b)=x_d-x_b$
between the size at division $x_d$  and the initial size $x_b$ of a cell 
is a random variable independent of $x_b$.
From this assumption it follows that 
\[
\tau(x_b)=\kappa^{-1}\ln((x_b+\Delta)/x_b).
\]
If $h(x)$ is the density distribution of $\Delta$, then 
\begin{equation}
\label{q-delta} 
q(x_b,a)=\kappa x_b e^{\kappa a} h\big(x_b e^{\kappa a}-x_b\big)
\end{equation}
 is the density of $\tau(x_b)$. 
We obtain a special case of our model if the density distribution of $\Delta$ is positive on the interval $(\dxb,\gxb)$.
According to experimental data from \cite{T-A} 
the coefficient of variation  $c_v$ of $\Delta$ for 
\textit{E. coli} is in the range of $0.17$ to $0.28$ depending on the different growth conditions.
We recall that   $c_v=\sigma/\mu$, where $\sigma$ is the standard deviation, and $\mu$ is the mean.

In \cite{Amir,Jafarpour} it is assumed  that
a cell with initial size $x_b$ attempts to
divide at a target size $x_d = f(x_b)$. Then the expected length of the cell cycle is $\tau_0(x_b)=\kappa^{-1}\ln(f(x_b)/x_b)$, but
$\tau_0(x_b)$  is additively perturbed by a symmetric random variable $\xi$, and finally $\tau(x_b)=\tau_0(x_b)+\xi$. 
If $h(a)$ is the density distribution of $\xi$, then $q(x_b,a)=h(a-\tau_0(x_b))$ is the density of $\tau(x_b)$. 
The authors assume in these papers  that $h$ has a normal distribution but in this case $\tau(x_b)$ can be negative therefore  
 a truncated normal distribution located in some interval $[-\varepsilon,\varepsilon]$   seems to be more suitable.  
 They also assume  that $f(x_b)=2x_b^{1-\alpha}x_0^{\alpha}$, $\alpha\in [0,1]$ and $x_0>0$.
 If $\alpha>0$, $\dxb=x_0e^{-\kappa\varepsilon /\alpha}$ and $\gxb=x_0e^{\kappa\varepsilon /\alpha}$, then we obtain a particular case of our model. 
 If $\alpha=0$, then $\tau_0\equiv \kappa^{-1}\ln 2$ and the length of cell cycle does not depend on $x_b$. In this case a daughter cell size is distributed 
in some neighbourhood of the initial mother cell size, so there is no minimum $\dxb$ and maximum size $\gxb$.

\subsection{Paradoxes of exponential growth}
\label{ss:paradoxes-exp-growth}
Models with exponential growth law can lead to some odd  mathematical results.
If the population starts with a single cell of size $x$,  cells from $n$th generation have size $x_n(t)=2^{-n}e^{\kappa t}x$ at time $t$. 
Since $ \dxb\le x_n(t)\le \gxb$,  population consists of a few generations at each time  and all cells in each generation have the same size.   
Usually the quotient $\gxb/\dxb$ is not too large. The initial size for \textit{E. coli}  under steady-growth conditions 
is  $x_b=2.32\, \pm\, 0.38\,\,\mu m$ (mean\,$\pm$\,SD) \cite{CSK}. 
Thus we can assume that in this case  $\gxb/\dxb<2$.
Then it is easy to check that if
\[
t\in \left(\frac {n+\log_2(\gxb/x)}{\kappa\log_2e},\frac {1+n+\log_2(\gxb/x)}{\kappa\log_2e}\right),
\]
the population consists only of cells from the $n$th generation, thus all cells have the same size and they cannot split in this time interval.
Consequently the size of the population never reaches an exponential (balanced) growth.
On this point we also observe that the large quotient $\gxb/\dxb$ helps the population to stabilize its growth,
which explains  why  
in the model with target size division  \cite{Jafarpour} it takes the population a longer time to reach its balanced growth for greater $\alpha$,
because $\gxb/\dxb=e^{2\kappa\varepsilon /\alpha}$.

The exponential growth law of cells  should be a little bit modified in order to achieve AEG.
For example it is enough to assume that $\kappa$ depends on the initial size $x_b$. But according to the experimental results,
the average growth rate does not depend on the initial size of cells. On the other hand,  even if  a population 
grows under perfect conditions the individual cells have different growth rates: 
the standard deviation  of the growth rate  is 15\%  of their respective means \cite{T-A}.  
Thus there is other factor called maturity,  which decides about the growth rate of an individual cell. The mathematical models based on the concept of maturity
were formulated in the late sixties \cite{LR,Rubinow}.
In such models the growth rate is identified  with maturation velocity $v$
which is constant during the life of  cell and
is inherited in a random way from mother to daughter cells.  

Rotenberg \cite{Rotenberg} considered a version of maturity models  with random jumps of $v$ during the cell cycle.
If we replace random jumps of $v$ by stochastic fluctuations of $\kappa$, 
we obtain a cell growth model described by a stochastic equation 
considered in the next subsection.

We consider here a simple generalization of our model assuming that $\kappa$ is a random variable with the distribution dependent on $x_b$. 
Let the function $r\mapsto  k(r|x_b)$ be the density of  $\kappa$.
The question is how to describe the joint distribution of age and initial size in this case.
Equations (\ref{eq1}) and (\ref{eq3}) remain the same and 
it is enough to derive a version of the boundary condition (\ref{eq2}).
Denote by $f(x;x_b,a)$ the density distribution of the random variable $\xi_a^{x_b}=x_be^{\kappa a}$.   
Then (\ref{eq2}) takes the form  
\begin{equation}
\label{eq2-s}
u(t,x,0)=4\iint\limits_X  f(2x;x_b,a)p(x_b,a)u(t,x_b,a)\,dx_b\,da. 
\end{equation}
It remains to find the function $f(x;x_b,a)$. We have 
\[
\Prob(x_be^{\kappa a}\le x)=\Prob\big(\kappa\le a^{-1}\ln(x/x_b)\big)
=\int_0^{a^{-1}\ln(x/x_b)}k(r|x_b)\,dr.
\]
Hence 
\[
f(x;x_b,a)=\frac1{ax}k(a^{-1}\ln(x/x_b)|x_b).
\]
At first glance formulae  (\ref{eq2}) and  (\ref{eq2-s}) differ significantly, but if we  
replace in (\ref{eq2-s}) the term $f(2x;x_b,a)$ by the delta Dirac $\delta_{S_a(x_b)}(x)$ we will receive
(\ref{eq2}).

\subsection{Stochastic growth of $x$}
\label{ss:stochastic-growth}
The size of a cell having initial size $x_b$ grows according to It\^o stochastic differential equation 
\begin{equation}
\label{stoch-grow}
d\xi_t^{x_b}= \kappa \xi_t^{x_b}\,dt+ \sigma(\xi_t^{x_b})\,dB_t, 
\end{equation}
where $B_t$, $t\ge 0$, is a one dimensional Wiener process (Brownian motion), and $\kappa>0$.
In \cite{I-B,PJI-B} the authors assume that  
$\sigma(x)=\sqrt{D}x^{\gamma}$, where $D>0$ and $\gamma\in (0,1)$. 
The  great strength of this formula  is that  equation (\ref{stoch-grow}) was 
intensively studied for such $\sigma$ and we can solve (\ref{stoch-grow}) and find various properties of solutions.
But there is one weak point: the size can go to zero and even solutions can be absorbed at zero.
To omit this problem we propose to assume that $\sigma\colon [\dxb,\infty)\to \mathbb R$ is a $C^1$-function and 
$\sigma(\dxb)=0$. Then $\xi_t^{x_b}>\dxb$ for $t>0$. It should be noted that solutions can decrease at some moments, i.e. a cell can shrink, 
but if the diffusion coefficient $\sigma$ is small, we observe exponential growth with small stochastic noise.    
If $f(x;x_b,a)$ is the density distribution of the random variable $\xi_a^{x_b}$, then    
the joint distribution of age and initial size $u(t,x_b,a)$ satisfies equations 
(\ref{eq1}), (\ref{eq3}), (\ref{eq2-s}).

\subsection{Models with asymmetric division and with slow-fast proliferation}
\label{ss:asymetric}
A lot of cellular populations are heterogeneous. For example, \textit{C. crescentus} has an asymmetric cell division;
\textit{B. subtilis} occasionally produces minicells;  melanoma cells  have slowly and quickly proliferating cells \cite{Perego}; 
and precursors of blood cells replicate and maturate going through the levels of morphological development \cite{Marciniak}.
It is difficult to find one universal model of the evolution of heterogeneous populations. Now we present a model of
the distribution of heterogeneous population based on similar
assumptions as the model presented in Section~\ref{s:model}.
We divide the population into a number of subpopulations. We assume that cells in the $i$th subpopulation 
grow according to the equation $x'=g_i(x)$ and  
their length of the cell cycle has the probability density distribution $q_i(x_b,a)$.
We also assume that $r_{ij}$ is the probability that a daughter of a cell from the $i$th subpopulation belongs to the $j$th subpopulation
and the daughter has initial size $\beta_{ij}x$, where $x$ is the size of the mother cell at division. 
As in Section~\ref{s:model} we introduce the function  
\[
p_i(x_b,a)=\frac{q_i(x_b,a)}{\int_a^{\infty} q_i(x_b,r)\,dr}
\]
and operators $P^{ij}_a$ which describe the relation between the density of the initial sizes of mother and daughter cells
satisfying the equation:
 \begin{equation*}
\int_{\dxb}^{\beta_{ij}\pi^i_ay}P^{ij}_af(x_b)\,dx_b=r_{ij}\int_{\dxb}^{y}f(x_b)\,dx_b.
 \end{equation*}
Then 
\begin{equation*}
P^{ij}_af(x_b)=\frac{r_{ij}}{\beta_{ij}}\frac{g(\pi^i_{-a}(x_b/\beta_{ij}))}{g(x_b/\beta_{ij})}f(\pi^i_{-a}(x_b/\beta_{ij})).
 \end{equation*}
We denote by $u_i(t,x_b,a)$ the number of individuals in the $i$th population  having initial size $x_b$ and age $a$ at time $t$. 
Then the system (\ref{eq1})--(\ref{eq3}) will be replaced by the following one
\begin{align}
\label{eq1-h}
&\frac{\partial u_i}{\partial t}(t,x_b,a)
 +\frac{\partial u_i}{\partial a}(t,x_b,a)
 =-p_i(x_b,a)u_i(t,x_b,a),\\
&u_j(t,x_b,0)=2\sum\limits_i\int_0^{\infty}P^{ij}_a(p_i(x_b,a)u_i(t,x_b,a))\,da, 
\label{eq2-h}\\
&u_i(0,x_b,a)=u_{i0}(x_b,a). 
\label{eq3-h}
\end{align}

In some cases of asymmetric division the size of daughter cells is not strictly determined and it is better to consider a model 
where the density $k(x_b|x_d)$ describes the distribution of the initial size of a daughter cell $x_b$
if the mother cell has the size $x_d$, see e.g. \cite{AK,GW,Heijmans,KDAT,RP}.

As an example of application of the model (\ref{eq1-h})--(\ref{eq3-h}) we consider \textit{C. crescentus} which  has asymmetric cell division into a "stalked" cell which can replicate and  a mobile "swarmer" cell which 
differentiates into a stalked cell after a short period of motility. Thus we have two subpopulations: the first -- stalked cells and the second -- swarmer cells.  Then $r_{ij}=1/2$ for $i=1,2$ and $j=1,2$. 
The stalked daughter has length of $0.56\,\pm\,0.04$ (mean\,$\pm$\,SD) of the mother cell \cite{CSK}. Hence we can assume that $\beta_{11}=\beta_{21}=0.56$ and 
$\beta_{12}=\beta_{22}=0.44$. If we assume that both stalked and swarmer cells have the same growth rate $\kappa$, i.e. $g_i(x)=\kappa x$, then
$q_2(x_b,a)=q_1(x_b,a-\rho)$, where  $\rho$ satisfies the formula $e^{\kappa \rho}=0.56/0.44$
and $q_1=q$ is given by  (\ref{q-delta}).

A model for the growth \textit{B. subtilis}  should be more advanced.  \textit{B. subtilis} can divide symmetrically to make two daughter cells (binary fission),
but some mutants split asymmetrically,  producing a single endospore, which can differentiate to a "typical" cell. Assume that the first population consists of typical cells and the second of minicells.  If $\mathfrak p$ is the probability of asymmetric fission, then 
$r_{11}=r_{21}=1-\mathfrak p+\mathfrak p/2=1-\mathfrak p/2$
and $r_{12}=r_{22}=\mathfrak p/2$. Some information on the size of minicells can be found in \cite{KH}. 

In a model which describes slowly and quickly proliferating cells we should assume that the length of the cell cycle of slowly proliferating cells is longer than 
in quickly proliferating cells and slowly proliferating cells also grow slower. Thus the sensible assumptions are:
$g_1(x)<g_2(x)$ and 
\[
\int_0^a q_1(x_b,r)\,dr <\int_0^aq_2(x_b,r)\,da\quad\text{for $a<\ga_1(x_b)$}.
\]
We should  also  assume that there is some transition between both subpopulations. Other model of the growth of the population with  slowly and quickly proliferating cells was recently studied in \cite{Vittadello}.

\section*{Acknowledgments}
This research was partially supported by 
the  National Science Centre (Poland)
Grant No. 2017/27/B/ST1/00100.

\end{document}